\documentclass[11pt,reqno]{amsart}
\usepackage{amsmath,amssymb,amsthm,upref,graphicx,mathrsfs,enumerate}
\usepackage{textcomp} 
\usepackage{array} 
\usepackage{textcomp} 
\usepackage[
  colorlinks=true,
  linkcolor=blue,
  citecolor=blue,
  urlcolor=blue,
  pagebackref]{hyperref}
\usepackage{enumitem}
\usepackage{hyperref}

\usepackage{color,soul}

\usepackage{amsthm}

\numberwithin{equation}{section}
\allowdisplaybreaks

\newtheorem{theorem}{Theorem}[section]
\newtheorem{prop}[theorem]{Proposition}
\newtheorem{lemma}[theorem]{Lemma}
\newtheorem{cor}[theorem]{Corollary}

\theoremstyle{definition}
\newtheorem{definition}[theorem]{Definition}

\newtheorem{remark}[theorem]{Remark}

\let \o=\omega

\begin{document}

\title[The modulus of $p$-variation and its applications]{The modulus of $p$-variation and its applications}

\author[G. H. Esslamzadeh, M. M. Goodarzi, M. Hormozi, M. Lind]{G. H. Esslamzadeh, M. M. Goodarzi, M. Hormozi, M. Lind}

\address[G. H. Esslamzadeh]{Department of Mathematics, Faculty of Sciences, Shiraz University, Shiraz 71454, Iran}
\email{esslamz@shirazu.ac.ir}
\address[M. M. Goodarzi]{Department of Mathematics, Faculty of Sciences, Shiraz University, Shiraz 71454, Iran}
\email{milad.moazami@gmail.com}

\address[M. Hormozi]{School of Mathematics, Institute for Research in Fundamental
Sciences (IPM), P. O. Box 19395-5746, Tehran, Iran
\&
Department of Mathematics, Institute for Advanced studies in Basic Sciences (IASBS), Zanjan, Iran  
}
\email{me.hormozi@gmail.com}
\address[M. Lind]{Department of Mathematics, Karlstad University, Universitetsgatan 2, 651 88 Karlstad, Sweden}
\email{martin.lind@kau.se}



\begin{abstract}
In this note, we introduce the notion of modulus of $p$-variation for a function of a real variable, and show that it serves in at least two important problems, namely, the uniform convergence of Fourier series and computation of certain $K$-functionals. To be more specific, let $\nu$ be a nondecreasing concave sequence of positive real numbers and $1\leq p<\infty$. Using our new tool, we first define a Banach space, denoted $V_p[\nu]$, that is intermediate between the Wiener class $BV_p$ and $L^\infty$, and prove that it satisfies a Helly-type selection principle. We also prove that the Peetre $K$-functional for the couple $(L^\infty,BV_p)$ can be expressed in terms of the modulus of $p$-variation. Next, we obtain equivalent sharp conditions for the uniform convergence of the Fourier series of all functions in each of the classes $V_p[\nu]$ and $H^\omega\cap V_p[\nu]$, where $\omega$ is a modulus of continuity and $H^\omega$ denotes its associated Lipschitz class. Finally, we establish optimal embeddings into $V_p[\nu]$ of various spaces of functions of generalized bounded variation. As a by-product of these latter results, we infer embedding results for certain symmetric sequence spaces.
\end{abstract}

\maketitle
\[\]

\section{Introduction}
The $n$th order variation of a function $f$ on an interval $[a,b]$ was originally defined by Lagrange \cite{Lag} to be $\upsilon(n,f):=\sup \sum_{j=1}^n |f(I_j)|$, where the supremum is taken over all finite collections $\{I_j\}_{j=1}^n$ of nonoverlapping intervals in $[a,b]$, $f(I_j)=f(\sup I_j)-f(\inf I_j)$ and $n$ is a positive integer. One attractive feature of $\upsilon(n,f)$ is that it is finite for each $n$ and any bounded function $f$. Thus, even though $f$ might \emph{not} be of bounded (Jordan) variation, one can still glean some insight from the growth rate of $\upsilon(n,f)$ with respect to $n$. Chanturiya \cite{Cha1} observed that this is indeed a nondecreasing and concave sequence of positive numbers, and coined the term \emph{modulus of variation} to refer to a generic sequence $\nu$ with such properties. He made extensive use of this notion in connection with the theory of Fourier series (see, e.g., \cite{Cha1,Cha3,Cha2,Cha4}). In particular, he obtained, among many other results, an estimate of the (uniform) distance between $f$ and the partial sums of its Fourier series in terms of moduli of variation and continuity. As a consequence, Chanturiya showed that estimates of Lebesgue and Oskolkov as well as Dini's criterion can be deduced from his more general estimate. The modulus of variation has found applications in other areas as well (see, e.g., \cite{Pych,CobosKruglyak,Chist2}). The reader is referred to \cite{Appell} and \cite{GNW} for more information on moduli of variation.

In this paper, we introduce the {\it modulus of $p$-variation of a function $f$} by considering $\ell_p$-norms of sequences of differences $\{f(I_j)\}_{j=1}^n$ as follows:
$$
\upsilon_p(n,f):=\sup \Big(\sum_{j=1}^n |f(I_j)|^p\Big)^{\frac1{p}}, \ \ \ 1\leq p<\infty.
$$
This is a measure of regularity of a function that will provide us with a natural tool to deal with rates of convergence, approximation of functions and embedding theorems. Note that for $p=1$, one gets Chanturiya's modulus of variation, so we deal with the case $p>1$.

Our motivation is twofold. The Harmonic-analytic motivation is to utilize moduli of $p$-variation in defining a new class of Banach spaces with nice ``compactness properties'', that encompasses the Wiener classes $BV_p$ as well as the Chanturiya classes $V[\nu]$ (see below), and to study certain aspects of them pertaining to the theory of Fourier series. We are here mainly concerned with the problem of uniform convergence for Fourier series.

The second motivation stems from an important problem in real interpolation theory, namely, characterization of Peetre's $K$-functionals (see next section for definition). These have proven to be very helpful in the study of interpolation spaces between two Banach spaces and interpolation of operators \cite{deVoreLorentz}. Computation of $K$-functionals is in general a difficult task and each new case provides some nontrivial information. For various couples of Banach spaces this has been done (see, e.g., \cite{deVoreLorentz,Pisier}). Our investigation reveals that the $K$-functional for the couple $(L^\infty,BV_p)$ is naturally linked to the modulus of $p$-variation.

We point out that our variational concept (i.e. the modulus of $p$-variation) specializes to functions defined on the unit circle. In the multivariate case, proposing Wiener-type $p$-variations is much less clear than in the univariate case. In particular, for $p=1$ there are several non-equivalent definitions of the notion of total variation, each of which being suitable for a specific purpose. In the very interesting recent works \cite{Brudnyi1,Brudnyi2}, the authors introduce a notion of $p$-variation that in a sense unifies the previous different theories. It might be worthwhile to try to define the modulus of $p$-variation for multivariate functions starting from the $p$-variation introduced in \cite{Brudnyi1}.

Before outlining the content of the paper, let us introduce the class $V_p[\nu]$ consisting of all bounded $2\pi$-periodic functions $f$ such that $\upsilon_p(n,f)\lesssim\nu(n)$, where $\nu$ is a modulus of variation and $1\leq p<\infty$ (by $A \lesssim B$ we mean $A\leq c B$, where $c$ is a positive constant). Notice that when $p=1$, we denote $V_p[\nu]$ simply by $V[\nu]$, which was first introduced and studied by Chanturiya (see, e.g., \cite{Cha1,Cha2}). If instead one takes $\nu(n)=1$, the Wiener class $BV_p$ \cite{Wi} is obtained.

In Section \ref{sec2} we recall some preliminary definitions and background material. In Section \ref{sec3} we first observe that $V_p[\nu]$ is a Banach space with respect to a suitable norm such that $BV_p\hookrightarrow V_p[\nu]\hookrightarrow L^\infty$, and then establish a Helly type theorem for this space (Theorem \ref{Helly}). As a consequence, a characterization of continuous functions in $V_p[\nu]$ in terms of their Fej\'{e}r means is obtained. In Section \ref{sec4} we prove that the Peetre $K$-functional for the couple $(L^\infty,BV_p)$ can be expressed in terms of the modulus of $p$-variation (Theorem \ref{KfunctTeo1}). Section \ref{sec5} is devoted to the statement and proof of Theorem \ref{convergence}, which presents several equivalent conditions for the uniform convergence of the Fourier series of all functions in the class $H^\omega \cap V_p[\nu]$, where $H^\omega$ is the Lipschitz class associated to a modulus of continuity $\omega$ (see the next section for precise definitions). In Section \ref{sec7} we establish optimal embeddings into $V_p[\nu]$ of several spaces of functions of generalized bounded variation (Theorem \ref{embedding}). The last section consists of two parts: in the first part we indicate how to apply Section \ref{sec7} to deduce embedding results for certain symmetric sequence spaces, and in the second part we prove equivalent sharp conditions for the uniform convergence of the Fourier series of all functions in $V_p[\nu]$. The Fourier coefficients in $V_p[\nu]$ are also estimated.

Throughout, we write $A\approx B$ to imply $A\lesssim B\lesssim A$. For a sequence $\{\gamma_k\}$ of real numbers we use Landau's notation $\gamma_k=o(k)$ to denote the condition $\lim\limits_{k\to \infty} \frac{\gamma_k}{k}=0$. The symbol $\gamma_k\downarrow$ ($\gamma_k\uparrow$) means that $\gamma_k$ is nonincreasing (nondecreasing). Also we define $\Delta(\gamma_k):=\gamma_k-\gamma_{k+1}$.

\section{Background and preliminaries}\label{sec2}

This section is devoted to some preliminary notions and definitions that are used in the sequel.

\subsection{Moduli of continuity and generalized variation}

By a {\it modulus of continuity} we mean a continuous, subadditive and nondecreasing function $\omega$ on the nonnegative real numbers such that $\omega(0)=0$. In particular,
$$
\omega(f, \delta):=\sup_{0\leq h\leq \delta} \sup_{x\in[0,2\pi]}|f(x+h)-f(x)|, \ \ \ \delta\geq0,
$$
is called the {\it modulus of continuity of the function f}. The symbol $H^\omega$ stands for the Lipschitz class of all $2\pi$-periodic functions for which $\omega(f,\delta)\lesssim \omega(\delta)$ as $\delta\rightarrow0^+$.

The {\it $L_p$-modulus of continuity of $f$} is defined as
$$
\omega_p(f,\delta):=
\sup_{0\leq h \leq\delta}\left(\int_0^{2\pi} |f(x+h)-f(x)|^pdx\right)^\frac1{p}, ~~~~~\qquad 1 \leq p<\infty.
$$
These moduli of continuity are well studied compared to moduli of variation. It is worth mentioning that $\upsilon_p(n,f)$ may be viewed as a ``variational'' counterpart of $\omega_p(f;\delta)$, and as such, $V_p[\nu]$ can be considered as an analogue of the generalized Lipschitz class $H_p^\omega$, consisting of all $2\pi$-periodic functions $f$ with $\omega_p(f,\delta)\lesssim \omega(\delta)$ as $\delta\rightarrow0^+$. This is another motivation to propound $V_p[\nu]$.

Let $\Phi=\{\phi_j\}_{j=1}^\infty$ be a sequence of increasing convex functions on the nonnegative reals such that $\phi_j(0)=0$ for all $j$.
We say that $\Phi$ is a $\Phi$-{\it sequence} if  for all $j$ and all $x>0$ we have $0<\phi_{j+1}(x)\leq\phi_j(x)$ and $\sum_{j=1}^\infty\phi_j(x)=\infty$. We denote the sequence of partial sums of the series $\sum_{j=1}^\infty\phi_j$ with $\{\Phi_n\}$. 
A real function $f$ on $[a,b]$ is said to be of {\it $\Phi$-bounded variation} if
$$
\text{Var}_\Phi(f)=\text{Var}_\Phi(f;[a,b])=\sup\sum_{j=1}^n\phi_j(|f(I_j)|)<\infty,
$$
where the supremum is taken over all finite collections $\{I_j\}_{j=1}^n$ of nonoverlapping subintervals of $[a,b]$.

We denote by $\Phi\text{BV}$ the linear space of all functions $f$ such that $cf$ is of $\Phi$-bounded variation for some $c>0$. This notion was originally introduced and studied in \cite{SW2}. Indeed, $\Phi\text{BV}$ with a suitable norm turns into a Banach space, in which the Helly selection principle holds.

\begin{remark}\label{rem1}
Let $\phi$ be an {\it Orlicz function}, that is, a continuous convex function on the nonnegative reals such that $\phi(x)>0$ for $x>0$, 
$\lim\limits_{x\rightarrow0^+}\frac{\phi(x)}{x}=0$ and
$\lim\limits_{x\rightarrow+\infty}\frac{\phi(x)}{x}=+\infty$. Then, by taking $\phi_j(x):=\phi(x)$ for all $j$ in the preceding definition, we get the Salem class $V_\phi$. If further we take $\phi(x)=x^p$ ($p>1$), the well-known Wiener class $BV_p$ is obtained. On the other hand, if $\phi$ is an Orlicz function, and if $\Lambda=\{\lambda_j\}_{j=1}^\infty$ is a $\Lambda$-{\it sequence} \cite{Wt1}, that is, a nondecreasing sequence of positive numbers such that $\sum_{j=1}^\infty\frac1{\lambda_j}=\infty$, by taking $\phi_j(x)=\frac{\phi(x)}{\lambda_j}$ for all $j$, we get the class $\phi\Lambda\text{BV}$ of functions of $\phi\Lambda$-bounded variation. If, further, we assume that $\phi(x)=x^p$ ($p\geq1$), we get the Waterman-Shiba class $\Lambda\text{BV}^{(p)}$. When $p=1$, we obtain the well-known Waterman class $\Lambda\text{BV}$ of functions of bounded $\Lambda$-variation. (See \cite{Appell} for more information.)
\end{remark}

Recall that a sequence $\{a_k\}$ of positive numbers is said to be {\it quasiconcave} if $\{a_k\}$ is nondecreasing, while $\{a_k/k\}$ is nonincreasing. Recall also that $\{a_k\}$ is said to be {\it concave} if $a_{k+1}+a_{k-1}\leq 2a_k$ for all $k$.

Suppose that $\nu$ is a modulus of variation.\footnote{As a convention we put $\nu(0)=0$.} For a bounded $\nu$, it is easily seen that $V_p[\nu]=BV_p$. Therefore, we assume throughout the paper that $\nu(k)\uparrow\infty$ as $k\rightarrow\infty$.

Let us here introduce a piece of notation that will be used in the sequel. For each positive integer $k$ we define
\begin{equation}\label{xi}
\varepsilon_p(k):=\big(\nu(k)^p-\nu(k-1)^p\big)^{\frac1{p}}.
\end{equation}
This sequence plays a significant role in our convergence results (Theorems \ref{convergence} and \ref{Unif2}), so in the following lemma we record some of its properties for future reference.

\begin{lemma}\label{epsionprop}
Let $\nu$ be a modulus of variation, $\varepsilon_p(k)$ be defined as above, and $1\leq p<\infty$. Then the following facts hold.

\medskip
\emph{(i)} $\frac{\nu(k)}{k^{\frac1{p}}}$ is nonincreasing if and only if $\nu(k)^p$ is quasiconcave if and only if $\varepsilon_p(k)\leq \frac{\nu(k)}{k^{\frac1{p}}}$;

\medskip
\emph{(ii)} $\varepsilon_p(k)$ is nonincreasing if and only if $\nu(k)^p$ is concave; in this case, $\frac{\nu(k)}{k^{\frac1{p}}}\rightarrow0$ implies that $\frac{\nu(k)}{k^{\frac1{p}}}$ is eventually decreasing;\footnote{By ``decreasing" we mean what is sometimes called ``strictly decreasing".} and

\medskip
\emph{(iii)} $\nu(k)-\nu(k-1)\lesssim \varepsilon_p(k)k^{\frac1{p}-1}$.

\end{lemma}
\begin{proof}
The first parts of (i) and (ii) are straightforward. To see the second part of (i), note that for each $k$ we have
\begin{align*}
\frac{\nu(k)}{k^{\frac1{p}}}\leq \frac{\nu(k-1)}{(k-1)^{\frac1{p}}} \ \ \ &\Longleftrightarrow \ \ \ k\nu(k)^p-\nu(k)^p\leq k\nu(k-1)^p\\[.1in]
&\Longleftrightarrow \ \ \ k\big(\nu(k)^p-\nu(k-1)^p\big)\leq \nu(k)^p\\[.1in]
&\Longleftrightarrow \ \ \ \varepsilon_p(k) \leq \frac{\nu(k)}{k^{\frac1{p}}}.
\end{align*}

To see the second part of (ii), we argue along the lines of the proof of \cite[Lemma 1]{Cha4}. Since $\varepsilon_p(k)$ is nonincreasing, for all $n$ we have
$$
n\big(\nu(n)^p-\nu(n-1)^p\big)\leq \sum_{k=1}^n \nu(k)^p-\nu(k-1)^p=\nu(n)^p,
$$
or $\varepsilon_p(n) \leq \frac{\nu(n)}{n^{\frac1{p}}}$. By (i) this means that $\frac{\nu(k)}{k^{\frac1{p}}}$ is nonincreasing. To obtain the desired result, we proceed by contradiction. If from some point on, $\frac{\nu(k)}{k^{\frac1{p}}}$ is not decreasing, there exists a sequence $\{n_k\}_{k=1}^\infty$ such that
$$
\frac{\nu(n_k)}{n_k^{\frac1{p}}}=\frac{\nu(n_k+1)}{(n_k+1)^{\frac1{p}}}, \ \ \ k\geq 1.
$$
But this means
$$
\frac{\nu(n_k)^p}{n_k}=\nu(n_k+1)^p-\nu(n_k)^p, \ \ \ k\geq 1,
$$
or
$$
\frac1{n_k}\sum_{m=1}^{n_k} \nu(m)^p-\nu(m-1)^p=\nu(n_k+1)^p-\nu(n_k)^p.
$$
Again since $\varepsilon_p(k)$ is nonincreasing, the latter equality implies that
$$
\frac{\nu(n)}{n^{\frac1{p}}}=\varepsilon_p(n)=\nu(1), \ \ \ 1\leq n\leq n_k+1.
$$
As $\{n_k\}$ is an infinite sequence, the preceding equality holds for all $n$. This contradicts the assumption that $\frac{\nu(k)}{k^{\frac1{p}}}\rightarrow0$.

To prove (iii), apply the inequality
\begin{equation}
s^p-t^p>pt^{p-1}(s-t), \ \ \ (s>t,\ p>1)
\end{equation}
with $s:=\nu(k)$, $t:=\nu(k-1)$, and use the fact that $k(\nu(k)-\nu(k-1))\lesssim \nu(k-1)$ to obtain
\begin{align*}
\frac{\varepsilon_p(k)}{k} & \gtrsim \big(\nu(k)-\nu(k-1)\big)^{\frac1{p}}\nu(k-1)^{1-\frac1{p}}k^{-1}\\[.1in]
&\gtrsim \big(\nu(k)-\nu(k-1)\big)^{\frac1{p}}\big(\nu(k)-\nu(k-1)\big)^{1-\frac1{p}}k^{1-\frac1{p}}k^{-1}\\[.1in]
&=\big(\nu(k)-\nu(k-1)\big)k^{-\frac{1}{p}}.
\end{align*}

\end{proof}
 
\begin{remark}\label{remark}
As stated in \cite[Theorem 5]{Cha1} (see also \cite[Lemma 3]{Chist}), a function $f$ is regulated if and only if $\lim\limits_{k\rightarrow\infty}\frac{\upsilon(k,f)}{k}=0$.  Note also that by Proposition \ref{coeff} below, the sequence $\left\{\nu(k)/k^{\frac1{p}}\right\}$ determines the convergence rate of the Fourier coefficients of functions in $V_p[\nu]$. With this in mind and to ensure that all functions in $V_p[\nu]$ are regulated, we assume throughout the paper that $\frac{\nu(k)}{k^{\frac1{p}}}\downarrow 0$. By Lemma \ref{epsionprop} this implies that $\nu^p$ is quasiconcave. In Section \ref{sec5} we further assume that $\nu^p$ is concave which is not too severe a restriction in our setting. We also assume that $\lim\limits_{t\rightarrow0^+}\frac{\omega(t)}{t}=+\infty$ to exclude some uninteresting cases. These assumptions play crucial roles in the proofs of Theorems \ref{Helly} and \ref{convergence}.
\end{remark}

\subsection{Peetre's $K$-functionals}
 Let $(X_0,\|\cdot\|_{X_0})$ and $(X_1,\|\cdot\|_{X_1})$ be two Banach spaces such that $X_1\hookrightarrow X_0$. For any $f\in X_0$, the $K$-{\it functional} for $X_0$ and $X_1$ is defined as
\begin{equation}
\nonumber
K(f,t; X_0, X_1)=\inf_{g\in X_1}\left(\|f-g\|_{X_0}+t\|g\|_{X_1}  \right),\ \ \  t\geq 0.
\end{equation}
Intuitively, the quantity $K(f,t; X_0, X_1)$ describes how well a function $f\in X_0$ can be approximated by a $g\in X_1$.

As mentioned in the introduction, the $K$-functionals associated to several couples of Banach spaces have already been characterized. A function $f$ is said to belong to the \emph{Sobolev space} $W^{1,p}$ if its first order weak derivative $f'\in L^p~~(1\le p<\infty)$. One sees immediately that $W^{1,p}\hookrightarrow BV_p$. Further, in \cite{Lind} it was proven that
\begin{equation}
    \label{p-ContModulus}
    K(f,t; BV_p, W^{1,p})\approx \omega_{1-1/p}(f;t^{p'})\quad(1<p<\infty,\quad p'=p/(p-1)),
\end{equation}
where $\o_{1-1/p}(f;\delta)$ is the so-called modulus of $p$-continuity of Terehin, see \cite{Lind} and the references therein. This provides a measure of how well functions in $BV_p$ can be approximated by smooth functions (in the sense of Sobolev). In \cite{CobosKruglyak}, the $K$-functional for the couple $(L^\infty,BV)$ was also characterized in terms of moduli of variation.

\section{The Banach space $V_p[\nu]$ and Helly's theorem}\label{sec3}

Helly's selection theorem states that any uniformly bounded sequence of nondecreasing bounded functions on $[a,b]$ has a pointwise convergent subsequence. In this section we establish a theorem of this type for $V_p[\nu]$ and then employ it in determining when a continuous function $f$ lies in $V_p[\nu]$. To this end, we first show that $V_p[\nu]$ is a linear space and if we define
$$
V_{p,\nu}(f):=\sup_{1\leq n<\infty}\frac{\upsilon_p(n,f)}{\nu(n)},
$$
then the quantity 
$$
\|f\|_{p,\nu}:=V_{p,\nu}(f)+\sup_{x\in[a,b]}|f(x)|<\infty
$$
defines  a norm on this space, with respect to which $V_p[\nu]$ is complete; this is the content of our first result.

\begin{prop}\label{Banach}
Let $\nu$ be a modulus of variation and $1\leq p <\infty$. Then $(V_p[\nu],\|\cdot\|_{p,\nu})$ is a Banach space.
\end{prop}
\begin{proof}
Let $\{I_j\}_{j=1}^k$ be a collection of nonoverlapping intervals in $[a,b]$ and $f,g\in V_p[\nu]$. Then by the triangle inequality we get
$$
\frac{\upsilon_p(k,f+g)}{\nu(k)}\leq \frac{\upsilon_p(k,f)}{\nu(k)}+\frac{\upsilon_p(k,g)}{\nu(k)}.
$$
As a result, we obtain $\|f+g\|_{p,\nu}\leq \|f\|_{p,\nu}+\|g\|_{p,\nu}$. Now, suppose $\|f\|_{p,\nu}=0$. Then obviously $f(a)=0$, and $|f(x)|=|f(x)-f(a)|\leq \upsilon_p(1,f)=0$ for every $x\in (a,b]$. As $\|cf\|_{p,\nu}=|c|\|f\|_{p,\nu}$ for every $c\in\mathbb{R}$, we have shown that $V_p[\nu]$ is a normed linear space.

To prove completeness of $V_p[\nu]$ in $\|\cdot\|_{p,\nu}$, let $\{f_n\}$ be a Cauchy sequence in $V_p[\nu]$. Then $\{f_n\}$ is uniformly Cauchy, i.e., for every $\epsilon>0$, there exists an integer $N$ such that for all $m,n\geq N$ and for each $x\in (a,b]$,
\begin{align*}
|f_n(x)-f_m(x)| 
\leq(1+\nu(1))\epsilon.
\end{align*}
This in turn implies the existence of a function $f$ on $[a,b]$ such that $f_n\rightarrow f$ uniformly (and hence pointwise) in $[a,b]$. We are going to show that $f\in V_p[\nu]$ and $\|f_n-f\|_{p,\nu}\rightarrow0$ as $n\rightarrow\infty$. To see this, notice that there exists an $M>0$ such that $\|f_n\|_{p,\nu}\leq M$ for all $n$. So, if $\{I_j\}_{j=1}^k$ is a collection of nonoverlapping intervals in $[a,b]$, then
$$
\Big(\sum_{j=1}^k |f(I_j)|^p\Big)^{\frac1{p}}=\lim_{n\rightarrow\infty}\Big(\sum_{j=1}^k |f_n(I_j)|^p\Big)^{\frac1{p}}\leq M\nu(k),
$$
which means $f\in V_p[\nu]$.

Finally, with $\epsilon$ and $\{I_j\}_{j=1}^k$ as above, we have
\begin{align*}
\Big(\sum_{j=1}^k |(f_n-f_m)(I_j)|^p\Big)^{\frac1{p}}&\leq \upsilon_p(k,f_n-f_m)\\[.1in]
&\leq \nu(k)\|f_n-f_m\|_{p,\nu}\\[.1in]
&<\nu(k)\epsilon.
\end{align*}
Thus,
$$
\Big(\sum_{j=1}^k |(f_n-f)(I_j)|^p\Big)^{\frac1{p}}=\lim_{m\rightarrow\infty}\Big(\sum_{j=1}^k |(f_n-f_m)(I_j)|^p\Big)^{\frac1{p}}\leq \nu(k)\epsilon.
$$
As $k$ was arbitrary, this implies $V_{p,\nu}(f_n-f)\leq\epsilon$. Therefore $\|f_n-f\|_{p,\nu}\rightarrow0$.
\end{proof}

In \cite[Theorem 1]{Chist}, Chistyakov established a far-reaching extension of Helly's theorem mentioned above to metric space valued sequences of functions of a real variable.
Such compactness principles are particularly useful so as to establish existence results (e.g. weak solvability of nonlinear PDE's). In order to prove Theorem \ref{Helly}, we need the following special case of \cite[Theorem 1]{Chist} that suits our purposes. 
\begin{lemma}\label{Chistthm}
Let $\{f_j\}$ be a pointwise bounded sequence on $[a,b]$ such that
\begin{equation*}
\mu(n) := \limsup_{j \to \infty} \upsilon(n,f_j)=o(n). 
\end{equation*}
Then there exists a subsequence of  $\{f_j\}$ that converges pointwise to a function $f$on $[a,b]$.
\end{lemma}
\begin{theorem}\label{Helly}
Let $\{f_n\}$ be a bounded sequence in $V_p[\nu]$ with $\|f_n\|_{p,\nu}\leq M$. Then there exists a subsequence $\{f_{n_j}\}$ of $\{f_n\}$,  converging pointwise to a function $f$ in $V_p[\nu]$ such that $\|f\|_{p,\nu}\leq 2M$.
\end{theorem}
\begin{proof}
First of all, by applying H\"{o}lder's inequality we obtain, for a generic function $g$, that
$$
\Big(\sum_{j=1}^k |g(I_j)|^p\Big)^{\frac1{p}} \leq \sum_{j=1}^k |g(I_j)| \leq \Big(\sum_{j=1}^k |g(I_j)|^p\Big)^{\frac1{p}}k^{1-\frac1{p}}.
$$
This implies
\begin{equation}\label{Hol}
\upsilon_p(n,g) \leq v(n,g) \leq \upsilon_p(n,g)k^{1-\frac1{p}},
\end{equation}
and hence, for each $n$ and each $k$, we get
$$
\frac{\upsilon(k,f_n)}{k}\leq\frac{\upsilon_p(k,f_n)}{k^{\frac1{p}}}\leq M\frac{\nu(k)}{k^{\frac1{p}}},
$$
which in turn yields
$$
\frac{\mu(k)}{k}\leq\frac{\mu_p(k)}{k^{\frac1{p}}}\leq M\frac{\nu(k)}{k^{\frac1{p}}},
$$
where
$$
\mu(k):=\limsup_{n\rightarrow\infty}\upsilon(k,f_n) \ \ \ \ \ \text{and} \ \ \ \ \ \mu_p(k):=\limsup_{n\rightarrow\infty}\upsilon_p(k,f_n).
$$
Also note that every bounded sequence in $V_p[\nu]$ is pointwise bounded. Since by hypothesis $\frac{\nu(k)}{k^{\frac1{p}}}\rightarrow0$, we may apply Lemma \ref{Chistthm} to find a subsequence $\{f_{n_j}\}$ and a function $f$ such that $f_{n_j}\rightarrow f$ pointwise, as $j\rightarrow\infty$.

Let $\epsilon>0$ be given and $m\in\mathbb{N}$. Then from the preceding statement we conclude  that for each collection $\mathcal{I}=\{I_k\}_{k=1}^m$ of nonoverlapping intervals in $[a,b]$, there exists a $j_0=j_0(\mathcal{I})$ such that
$$
\Big|\Big(\sum_{k=1}^m |f(I_k)|^p\Big)^{\frac1{p}}-\Big(\sum_{k=1}^m |f_{n_{j_0}}(I_k)|^p\Big)^{\frac1{p}}\Big|<\epsilon\cdot\nu(m).
$$
As a result, we get
\begin{align*}
\Big(\sum_{k=1}^m |f(I_k)|^p\Big)^{\frac1{p}}&<\epsilon\cdot\nu(m)+\Big(\sum_{k=1}^m |f_{n_{j_0}}(I_k)|^p\Big)^{\frac1{p}}\\[.1in]
&\leq\epsilon\cdot\nu(m)+\upsilon_p(m,f_{n_{j_0}})\\[.1in]
&\leq (\epsilon+M)\nu(m),
\end{align*}
whence $V_{p,\nu}(f)\leq M$. Finally, since $|f_{n_j}(x)|\leq \|f_{n_j}\|_{p,\nu}\leq M$ for all $j$ and all $x\in[a,b]$, and $f_{n_j}(x)\rightarrow f(x)$ as $j\rightarrow\infty$, it follows that $\sup_{x\in[a,b]}|f(x)|\leq M$. Therefore $\|f\|_{p,\nu}\leq 2M$, as desired.
\end{proof}

The preceding theorem enables us to describe elements of $V_p[\nu]$ in terms of Fej\'{e}r means. For each $n$, the  $n$th order {\it Fej\'{e}r kernel} is defined by
$$
K_n(t):=\frac{2}{n+1}\left(\frac{\sin\frac1{2}(n+1)t}{2\sin\frac1{2}t}\right)^2,
$$
and  the $n$th order {\it Fej\'{e}r mean} of a function $f$ is given by 
$$
F_n(x):=\frac1{\pi}\int_{-\pi}^{\pi}f(x+t)K_n(t)dt.
$$

The following proposition is analogous to a result of Zygmund \cite[p. 138]{Zygmund}.

\begin{prop}
Let $f$ be continuous and $2\pi$-periodic. Then $f\in V_p[\nu]$ if and only if the  Fej\'{e}r means of $f$ form a bounded sequence in $V_p[\nu]$.
\end{prop}
\begin{proof}[Proof]
Suppose the sequence $\{F_n\}$  of  Fej\'{e}r means of $f$ is  bounded in $V_p[\nu]$. By Fej\'{e}r's theorem \cite[p. 89]{Zygmund} it converges pointwise to $ f$. Furthermore, Theorem \ref{Helly} ensures  the existence of a subsequence $\{F_{n_j}\}$ and a function $g$ in $V_p[\nu]$ such that $F_{n_j}\rightarrow g$ pointwise as $j\rightarrow\infty$. Evidently $f=g$, and hence $f\in V_p[\nu]$.

Conversely, since $K_n$ is positive and $\int_{-\pi}^{\pi}K_n(t)dt=\pi$, for any $f\in V_p[\nu]$ and any collection $\{I_k\}_{k=1}^m$ of nonoverlapping intervals, we get
\begin{align*}
\Big(\sum_{k=1}^m |F_n(I_k)|^p\Big)^{\frac1{p}} &=\Big(\sum_{k=1}^m \Big|\frac1{\pi}\int_{-\pi}^{\pi}f(I_k+t)K_n(t)dt\Big|^p\Big)^{\frac1{p}}\\[.1in]
& \leq\Big(\sum_{k=1}^m \Big(\frac1{\pi}\int_{-\pi}^{\pi}|f(I_k+t)|K_n(t)dt\Big)^p\Big)^{\frac1{p}}\\[.1in]
& \leq\Big(\sum_{k=1}^m \frac1{\pi}\int_{-\pi}^{\pi}|f(I_k+t)|^pK_n(t) dt\Big)^{\frac1{p}}\\[.1in]
& =\Big(\frac1{\pi}\int_{-\pi}^{\pi}\sum_{k=1}^m|f(I_k+t)|^pK_n(t) dt\Big)^{\frac1{p}}\\[.1in]
& \leq \Big(\frac1{\pi}\int_{-\pi}^{\pi}\upsilon_p(m,f)^pK_n(t) dt\Big)^{\frac1{p}}\\[.1in]
&=\upsilon_p(m,f)\leq V_{p,\nu}(f)\nu(m),
\end{align*}
where the second inequality is a consequence of Jensen's inequality. In particular we get $V_{p,\nu}(F_n)\leq V_{p,\nu}(f)$ for all $n$. Finally, since $F_n(a)\rightarrow f(a)$, the sequence $\{F_n\}$ is bounded in $V_p[\nu]$.
\end{proof}


\section{The modulus of $p$-variation as a $K$-functional}\label{sec4}

The aim of this section is to calculate $K(f,t; L^\infty, BV_p)$ in terms of the modulus of $p$-variation of $f$. Throughout this section, we consider functions defined on $[0,1]$ and we denote their $p$-variations by
$$
{\rm Var}_p(f)={\rm Var}_\Phi(f;[0,1]),
$$
where $\phi_j(x)=x^p$ for all $j$ (see Remark \ref{rem1}). Obviously, $BV_p\hookrightarrow L^\infty$and it makes sense to investigate how well, bounded functions can be approximated by functions in $BV_p$. We consider only continuous functions which is not a very severe restriction. 

The following simple lemma will be useful in the proof of Theorem \ref{KfunctTeo1}. 
\begin{lemma}
    \label{approxLemma}
    Let $f$ be a function on $[0,1]$ and let 
    $0=x_0<x_1<...<x_M=1$. If $g_M$ is the piecewise linear function that interpolates $f$ at the points $(x_i,f(x_i))~~(0\le i\le M)$, then 
    \begin{equation}
        \nonumber
        {\rm Var}_p(g_M)\le \upsilon_p(M,f).
    \end{equation}
\end{lemma}
\begin{proof}
Let $\{x_{i_k}\}$ be the subset of $\{x_i\}$ consisting of points of local extremum of $g_M$. It is easy to see that
$$
{\rm Var}_p(g_M)=\left(\sum_k|f(x_{i_{k+1}})-f(x_{i_k})|^p\right)^{1/p}\le \upsilon_p(M,f),
$$
since the sum extends over at most $M$ terms.
\end{proof}

\begin{theorem}
\label{KfunctTeo1}
Let $f$ be continuous and $1\le p<\infty$. Then
\begin{equation}
    \label{Kfunct1}
    K(f,t;L^\infty, BV_p)\approx t \upsilon_p\left(\left[t^{-1}\right]^p,f\right).
\end{equation}
\end{theorem}
\begin{proof}
By definition, we have
$$
K(f,t; L^\infty, BV_p)=\inf_{g\in BV_p}\left(\|f-g\|_{L^\infty}+t{\rm Var}_p(g)\right).
$$
(Note although that ${\rm Var}_p(\cdot)$ is simply a semi-norm on $BV_p$.) We shall first prove that
the right-hand side of (\ref{Kfunct1}) is dominated by the left-hand side. For any $M\in\mathbb{N}$ and any $g\in BV_p$, we have
\begin{eqnarray}
\nonumber
t\upsilon_p(M,f)&\le& t\upsilon_p(M,f-g)+t\upsilon_p(M,g)\\
\nonumber
&\le&tM^{1/p}\|f-g\|_{L^\infty}+t{\rm Var}_p(g).
\end{eqnarray}
Taking $M=\left[t^{-1}\right]^p$ we obtain
$$
t\upsilon_p\left(\left[t^{-1}\right]^p, f\right)\le\|f-g\|_{L^\infty}+t{\rm Var}_p(g)
$$
for any $g\in BV_p$, and it follows that
$$
t\upsilon_p\left(\left[t^{-1}\right]^p, f\right)\le K(f,t;L^\infty,BV_p).
$$
We proceed with the opposite inequality. Fix $t>0$ and let $M=\left[t^{-1}\right]^p$ as above. We shall construct a function $g_M\in BV_p$ such that
\begin{equation}
\label{KfunctEst1}
\|f-g_M\|_{L^\infty}+t{\rm Var}_p(g)\lesssim t\upsilon_p(M,f). 
\end{equation}
The function $g_M$ will be a first-order spline interpolant of $f$ with knots $\{x_i\}$. The set of knots depends on both $f$ and $M$, i.e. $g_M$ will be obtained through a nonlinear approximation procedure. In spirit, our construction is somewhat similar to the one given in \cite{CobosKruglyak}.

Set $x_0=0$ and consider two cases.

{\bf Case I:} Assume that we can obtain $M$ additional knots by the following construction: for $j=0,1,...,M-2$ define
\begin{equation}
    \label{nonTrivial}
     x_{j+1}=\min\left\{t:\, x_j<t\le 1\quad{\rm and}\quad |f(t)-f(x_j)|=\frac{\upsilon_p(M,f)}{M^{1/p}}\right\},
\end{equation}
and $x_M=1$. Set $J_i=[x_i,x_{i+1}]$ for $0\le i\le M-1$. Continuity of $f$ ensures that the above construction makes sense. Per definition,
\begin{equation*}
\label{KfunctEst2}
\max_{t\in J_i}|f(t)-f(x_i)|\le|f(x_{i+1})-f(x_i)|=\frac{\upsilon_p(M,f)}{M^{1/p}},
\end{equation*}
for $j=0,...,M-2$. We shall prove that
\begin{equation}
    \label{KfunctEst3}
    \max_{t\in J_{M-1}}|f(t)-f(x_{M-1})|\le \frac{\upsilon_p(M,f)}{M^{1/p}}.
\end{equation}
Assume that (\ref{KfunctEst3}) is false, then there exists $\xi\in(x_{M-1},x_M]$ such that 
$$
|f(\xi)-f(x_{M-1})|>\frac{\upsilon_p(M,f)}{M^{1/p}}.
$$
Set $I=[x_{M-1},\xi]$, then
\begin{eqnarray}
\nonumber
|f(I)|^p+\sum_{j=0}^{M-2}|f(J_i)|^p&=&|f(I)|^p+(M-1)\frac{\upsilon_p(M,f)^p}{M}\\
\nonumber
&>&\frac{\upsilon_p(M,f)^p}{M}+(M-1)\frac{\upsilon_p(M,f)^p}{M}\\
\nonumber
&=&\upsilon_p(M,f)^p
\end{eqnarray}
which is a contradiction.
Define now $g_M$ to be the piecewise linear function that interpolates $f$ at the points $(x_i,f(x_i))$ for $i=0,1,...,M$. By Lemma \ref{approxLemma},
\begin{equation}
    \label{pVarEst}
    {\rm Var}_p(g_M)\le \upsilon_p(M,f).
\end{equation}
At the same time,
$$
\|f-g_M\|_{L^\infty}=\max_j\|f-g_M\|_{L^\infty(J_i)}
$$
and for each $t\in J_i$ we have by definition
\begin{eqnarray}
\nonumber
|f(t)-g_M(t)|&=&\left|f(t)-f(x_j)-\frac{f(x_{j+1})-f(x_j)}{x_{j+1}-x_j}(t-x_j)\right|\\
\nonumber
&\le&|f(t)-f(x_j)|+|f(x_{j+1})-f(x_j)|\\
\nonumber
&\le& 2\frac{\upsilon_p(M,f)}{M^{1/p}}
\end{eqnarray}
Hence,
\begin{equation}
    \label{LInfEst}
    \|f-g_M\|_{L^\infty}\le\frac{2\upsilon_p(M,f)}{M^{1/p}}.
\end{equation}
By (\ref{pVarEst}) and (\ref{LInfEst}), we have
\begin{align*}
    \nonumber
    \|f-g_M\|_{L^\infty}+t{\rm Var}_p(g_M)&\le\frac{2\upsilon_p(M,f)}{M^{1/p}}+t\upsilon_p(M,f)\\[.1in]
    &\le 5t\upsilon_p(M,f)
\end{align*}
since $M^{-1/p}\le 2t$.

{\bf Case II:} Assume that the construction  (\ref{nonTrivial}) only yields $N<M$ additional knots. This case includes particularly the situation when 
\begin{equation}
\nonumber
\max_{0\le t\le1}|f(t)-f(0)|<\frac{\upsilon_p(M,f)}{M^{1/p}},
\end{equation}
and only one additional knot $x_1=1$ is obtained.
We define $g_M$ to be the linear interpolant at the $N+1$ point $(x_i,f(x_i))$. It is clear that ${\rm Var}_p(g_M)\le \upsilon_p(N,f)\lesssim\upsilon_p(M,f)$ and as above we have $\|f-g_M\|_{L^\infty}\le 2M^{-1/p}\upsilon_p(M,f)$.
The proof of (\ref{KfunctEst1}) is concluded as in Case I above.
\end{proof}

\begin{remark}
Observe that
$$
BV_p\hookrightarrow BV_q\quad(1\le p<q\le \infty),
$$
where we use the convention $BV_\infty=L^\infty$. It would be of some interest to calculate $K(f,t;BV_q,BV_p)$, and then Theorem \ref{KfunctTeo1} would simply correspond to the case $q=\infty$. Simple considerations suggest that
\begin{equation*}
    \label{conjecture}
    K(f,t; BV_q,BV_p)\approx t\upsilon_p\left(\left[t^{-1}\right]^{1/s},f\right),
\end{equation*}
where $s=1/p-1/q$, but we do not have a proof for this.
\end{remark}


\section{The uniform convergence of Fourier series }\label{sec5}

Let $\omega$ be a modulus of continuity, $\nu$ and $\nu^p$ be moduli of variation, and $1\leq p <\infty$. In this section we aim to establish necessary and sufficient conditions to guarantee uniform convergence of the Fourier series of every function in the class $H^\omega\cap V_p[\nu]$.

\medskip
Throughout this section, we use the following notation:
\begin{equation}\label{rho}
\rho(n):=\min_{1 \leq r \leq n-1} \left\{\omega\left(\frac{1}{n}\right) \sum_{k=1}^{r} \frac{1}{k}+\sum_{k=r+1}^{n-1}\frac{\nu(k)}{k^{1+\frac{1}{p}}}\right\}.
\end{equation}
\begin{remark}
Suppose the minimum in \eqref{rho} is attained at $r=\theta=\theta(n)$. Then by Lemma \ref{epsionprop}(ii), there exists some integer $N$ such that the sequence $\frac{\nu(n)}{n^{\frac1{p}}}$ is decreasing for $n\ge N$, and Lemma \ref{thetaprop} implies that $\theta(n)$ is uniquely determined for $n\geq N$. For $n<N$, we take $\theta(n)$ to be the smallest $r$ realizing the minimum in \eqref{rho}.
\end{remark}
Define
\begin{equation*}\label{sigma}
\sigma(n):=\omega\left(\frac{1}{n}\right) \sum_{k=1}^{\theta} \frac{1}{k}+\sum_{k=\theta+1}^{n-1}\frac{\varepsilon_p(k)}{k},
\end{equation*}
\begin{equation*}
\tau(n):=\omega\left(\frac{1}{n}\right) \sum_{k=1}^{\theta} \frac{1}{k}+\sum_{k=\theta+1}^{n-1}\frac{\nu(k)-\nu(k-1)}{k^{\frac{1}{p}}},
\end{equation*}
and
\begin{equation*}
\eta(n):=\omega\left(\frac{1}{n}\right) \sum_{k=1}^{\theta} \frac{1}{k}+\sum_{k=\theta+1}^{n-1}\Delta\Big(\frac1{k^{\frac{1}{p}}}\Big)\nu(k).
\end{equation*}

Our main result of this section may be formulated as follows.
\begin{theorem} \label{convergence}
Let $\omega$, $\nu$ and $p$ be as above. Then the following conditions are equivalent.
\begin{enumerate}

\medskip
\item[\emph{(i)}] The Fourier series of every function in the class $H^\omega\cap V_p[\nu]$ converges uniformly;

\medskip
\item[\emph{(ii)}] $\lim\limits_{n\rightarrow \infty} \sigma(n)=0$;

\medskip
\item[\emph{(iii)}] $\lim\limits_{n\rightarrow \infty} \rho(n)=0$;

\medskip
\item[\emph{(iv)}] $\lim\limits_{n\rightarrow \infty} \tau(n)=0$;

\medskip
\item[\emph{(v)}] $\lim\limits_{n\rightarrow \infty} \eta(n)=0$.
\end{enumerate}
\end{theorem}
The proof of this result is inspired by the method given in \cite {Cha2}, but essential modifications involving construction of new auxiliary functions are required to control the distortions introduced by $p$th powers. In order to prove this result, several auxiliary lemmas of technical nature are required, and we collect them in the following subsection.    

\subsection{Technical lemmas} Let us begin with recalling Abel's transformation
\begin{equation}\label{Abel}
\sum_{k=1}^n x_ky_k=\sum_{k=1}^{n-1}\Delta(x_k)\sum_{j=1}^{k}y_j+x_n\sum_{j=1}^ny_j,
\end{equation}
which proves useful at several points in what follows.
\begin{lemma}\label{thetaprop}
With the above notation we have the following.
\begin{enumerate}

\medskip
\item[\emph{(a)}] $\frac{\nu(\theta+1)}{(\theta+1)^\frac1{p}} \leq \omega\left(\frac{1}{n}\right)\leq \frac{\nu(\theta)}{\theta^\frac1{p}}$ whenever $\theta<n-1$;

\medskip
\item[\emph{(b)}] $\omega\left(\frac{1}{n}\right)< \frac{\nu(\theta)}{\theta^\frac1{p}}$ whenever $\theta=n-1$.
\end{enumerate}
\end{lemma}
\begin{proof} (a) Let $\theta< n-1$. Then
\begin{equation*}\label{itheta}
\omega\left(\frac{1}{n}\right)\sum_{k=1}^{\theta} \frac{1}{k} + \sum_{k=\theta+1}^{n-1}\frac{\nu(k)}{k^{1+\frac1{p}}} \leq  \omega\left(\frac{1}{n}\right)\sum_{k=1}^{r} \frac{1}{k} + \sum_{k=r+1}^{n-1}\frac{\nu(k)}{k^{1+\frac1{p}}}
\end{equation*}
where $r=\theta-1$ or $\theta+1$. This yields
\begin{equation*}\label{cond-theta}
\frac{\nu(\theta+1)}{(\theta+1)^\frac1{p}} \leq \omega\left(\frac{1}{n}\right)\leq \frac{\nu(\theta)}{\theta^\frac1{p}}.
\end{equation*}
Likewise, one may verify (b). 
\end{proof}

\begin{lemma}\label{definitionphipsi}
There exist sequences $\varphi$ and $\psi$ of positive integers with the following properties.
\begin{enumerate}

\medskip
\item[\emph{(a)}] $\frac{\varphi(n)}{n}\rightarrow 0$ and $\frac{2n}{\varphi(n)}\leq n\omega\left(\frac1{n}\right)\leq \varphi(n)\uparrow\infty$ as $n\rightarrow\infty$; (See Remark \ref{remark}.)

\medskip
\item[\emph{(b)}] \vspace{.1cm} $\sum\limits_{k=\varphi(n)+1}^{n-1} \frac{\varepsilon_p(k)}{k}\rightarrow 0$ as $n\rightarrow\infty$;

\medskip
\item[\emph{(c)}] \vspace{.1cm} $4 \ \nu(\psi(n)^2)<\nu\big(\min\{\varphi(n),\theta(n)\}\big)$ and $\omega\left(\frac1{n}\right)^{-1}\geq \psi(n)\uparrow\infty$ as $n\rightarrow\infty$.
\end{enumerate}
\end{lemma}
\begin{proof}
For large enough $n$, let $\beta_n$ be the greatest integer such that $\beta_n\leq n\varepsilon_p(\beta_n)$.
The $\beta_n$ are well-defined since $\nu(k)\uparrow\infty$ and $\varepsilon_p(k)\downarrow 0$ as $k\rightarrow\infty$, which as well implies that $\beta_n \uparrow\infty$ as $n\rightarrow\infty$. Thus, we see that
$\frac{\beta_n}{n}\leq \varepsilon_p(\beta_n)\rightarrow 0$

Existence of a sequence $\varphi$ satisfying (a) is guaranteed by defining $\varphi(n)$ to be the smallest integer such that 
$$
\varphi(n)\geq\max \left\{\beta_n,n\omega\left(\frac1{n}\right),2\omega\left(\frac1{n}\right)^{-1}\right\}.
$$
On the other hand, from the  definition of $\beta_n$ it follows that $n\varepsilon_p(\varphi(n)+1)< \varphi(n)+1$, since $\beta_n\leq \varphi(n)$. By using this fact, we obtain
\begin{align*}\label{rr}
\sum_{k=\varphi(n)+1}^{n-1} \frac{\varepsilon_p(k)}{k}&\leq\varepsilon_p(\varphi(n)+1)\sum_{k=\varphi(n)+1}^{n-1} \frac{1}{k}\\[.1in]
&<\frac{\varphi(n)+1}{n} \sum_{k=\varphi(n)+1}^{n-1} \frac{1}{k}\\[.1in]
&\lesssim \frac{\varphi(n)}{n} \ln \frac{n}{\varphi(n)},
\end{align*}
which implies (b), since $\frac{n}{\varphi(n)}\rightarrow\infty$ by (a). To see (c), as above let $n$ be large enough such that $4\nu(1)<\nu\big(\min\{\varphi(n),\theta(n)\}\big)$, and define $\gamma_n$ to be the greatest integer such that $4\nu(\gamma_n^2)<\nu\big(\min\{\varphi(n),\theta(n)\}\big)$. Then evidently, we have $\gamma_n\uparrow\infty$ since $\varphi(n)\uparrow\infty$ and $\theta(n)\uparrow\infty$. Finally, by taking $\psi(n)$ to be the largest integer such that
$$
\psi(n)\leq \min \left\{\gamma_n,\omega\left(\frac1{n}\right)^{-1}\right\},
$$
part (c) is obtained.
\end{proof}

{\bf Notation.} In the sequel, we shall use the following symbols. For every $n$, $\mu(n):=\min\{\varphi(n),\theta(n)\}$ and $M(n):=\max\{\varphi(n),\theta(n)\}$.

Let $\{\sigma(n_i)\}$ be a subsequence of $\{\sigma(n)\}$ such that $\lim\limits_{i\rightarrow\infty}\sigma(n_i)=\limsup\limits_{n \to \infty} \sigma(n)$, and define
$$
\bar{\sigma}(n):=\omega\left(\frac{1}{n}\right) \sum_{k=\psi(n)}^{\mu(n)} \frac{1}{k}+\chi(n)\sum_{k=\mu(n)+1}^{M(n)}\frac{\varepsilon_p(k)}{k},
$$
where the sequence $\chi$ is defined as follows: $\chi(n)=1$ when $\varphi(n)-\theta(n)>0$, whereas $\chi(n)=0$ when $\varphi(n)-\theta(n)\leq 0$.

\begin{lemma}\label{lemma 4.4}
With the above notation, the following two statements hold true.
\begin{enumerate}

\medskip
\item[\emph{(a)}] $\lim\limits_{i\rightarrow \infty} \bar{\sigma}(n_i)=\limsup\limits_{n \to \infty} \sigma(n)$;

\medskip
\item[\emph{(b)}] $\sigma(n)-\bar{\sigma}(n)\lesssim \omega(\frac{1}{n}) \ln\big(\omega(\frac{1}{n})^{-1}\big) \lesssim  \bar{\sigma}(n)$.
\end{enumerate}
\end{lemma}
\begin{proof}
(a) By Lemma \ref{definitionphipsi}(c) we see that
\begin{equation}\label{psilnrelation}
\omega\left(\frac{1}{n}\right) \sum_{k=1}^{\psi(n)} \frac{1}{k} \lesssim \omega\left(\frac{1}{n}\right) \ln\left( \omega\Big(\frac{1}{n}\Big)^{-1}\right)\rightarrow 0.
\end{equation}
On the other hand, Lemma \ref{definitionphipsi}(b) says that
\begin{equation}\label{sumo(1)}
\sum\limits_{k=\varphi(n)+1}^{n-1} \frac{\varepsilon_p(k)}{k}\rightarrow 0.
\end{equation}
With \eqref{psilnrelation} and \eqref{sumo(1)} in mind, (a) is proven when $\theta(n)< \varphi(n)$.

Now suppose that $\varphi(n)\leq\theta(n)$. By using Abel's transformation, one can observe that 
$$
\sum\limits_{k=\varphi(n)+1}^{\theta(n)} \frac{\nu(k)}{k^{1+\frac1{p}}}= \sum\limits_{k=\varphi(n)+1}^{\theta(n)} \frac{\nu(k)-\nu(k-1)}{k^\frac1{p}}-\frac{\nu(\theta(n))}{(\theta(n)+1)^{\frac{1}{p}}}+\frac{\nu(\varphi(n))}{(\varphi(n)+1)^{\frac{1}{p}}}.
$$
So we get
\begin{align*}
\omega\Big(\frac{1}{n}\Big) \sum_{k=\varphi(n)+1}^{\theta(n)}\frac{1}{k} &\overset{\text{(Lemma \ref{thetaprop})}}{\leq} \frac{\nu(\theta(n))}{\theta(n)^{\frac{1}{p}}} \sum_{k=\varphi(n)+1}^{\theta(n)}\frac{1}{k}\\[.1in]
&\hspace{.7cm}\leq \sum_{k=\varphi(n)+1}^{\theta(n)}\frac{\nu(k)}{k^{1+\frac{1}{p}}}\\[.1in]
&\overset{\text{(Lemma \ref{epsionprop}(iii))}}{\lesssim} \frac{\nu(\varphi(n))}{(\varphi(n)+1)^{\frac{1}{p}}}+\sum\limits_{k=\varphi(n)+1}^{n-1} \frac{\varepsilon_p(k)}{k} \overset{\text{\eqref{sumo(1)}}}\longrightarrow 0, \ \ \text{as} \ \ n\rightarrow\infty,
\end{align*}
where the second inequality is a result of the fact that $\frac{\nu(k)}{k^{\frac{1}{p}}}\downarrow$. Thus, part (a) is proven when $\varphi(n)\leq\theta(n)$.

\medskip
(b) Due to Lemma \ref{thetaprop}(a) and  Lemma \ref{definitionphipsi}(a), we have 

\begin{equation*}\label{muomega}
\mu(n)+1 \geq \frac{1}{\omega\left(\frac{1}{n}\right)}.
\end{equation*}
Now $\psi(n) \leq \sqrt{\mu(n)}$ implies 
\begin{align*}
\bar{\sigma}(n)&> \omega\left(\frac{1}{n}\right) \sum_{k=\psi(n)}^{\mu(n)} \frac{1}{k}> \omega\left(\frac{1}{n}\right) \sum_{k=\big[\sqrt{\mu(n)}\big]+1}^{\mu(n)} \frac{1}{k}\\[.1in]
&\gtrsim\omega\left(\frac{1}{n}\right) \ln (\mu(n)+1)\gtrsim \omega\left(\frac{1}{n}\right) \ln\left( \omega\Big(\frac{1}{n}\Big)^{-1}\right).
\end{align*}

Without much difficulty, it can also be verified that
$$
\sigma(n)-\bar{\sigma}(n) \lesssim \omega\left(\frac{1}{n}\right) \ln\left( \omega\Big(\frac{1}{n}\Big)^{-1}\right),
$$
and we leave it to the reader.
\end{proof}
By using Lemma \ref{definitionphipsi}, Remark \ref{remark} and similar arguments in \cite[P. 489]{Cha2}, the following lemma is proven. 
\begin{lemma}\label{Lemma 4.5}
There exists an increasing subsequence $\{l_k\}$ of $\{n_i\}$ with the following properties.
\begin{enumerate}

\medskip
\item[\emph{(a)}] $\frac{\varphi(l_k)}{l_k}<\frac{\psi(l_{k-1})}{l_{k-1}}$.

\medskip
\item[\emph{(b)}] $\omega(\frac{1}{l_k}) \leq \varepsilon_p(\varphi(l_{k-1}))$.

\medskip
\item[\emph{(c)}] $\varphi(l_{k-1}) <\psi(l_k)$.

\medskip
\item[\emph{(d)}] $\psi(l_{k-1})\omega\left(\frac{1}{l_k}\right)<\frac1{36\pi}\bar{\sigma}(l_{k-1})$.

\medskip
\item[\emph{(e)}]
$A_k\ln\left(A_k^{-1}\sum_{i=1}^{k-1}\omega\left(\frac{1}{l_i}\right)\varphi(l_i)\right)<\bar{c} \ \bar{\sigma}(l_k), \ \ \ \text{where} \ \ \ A_k=\frac{l_{k-1}}{l_k}\omega\left(\frac{1}{l_{k-1}}\right)$.
\end{enumerate}
\end{lemma}
\begin{lemma}[\cite{Osk}]\label{Oskol}
Let  $\{\Delta_k\}$ be a sequence of disjoint intervals in $[0,2\pi]$  and $\{g_k\}$ be a sequence of $2\pi$-periodic functions such that $g_k(x)=0$ for $x \in [0,2\pi] \setminus \Delta_k$. If the function $g$ is defined by $g= \sum\limits_{k=1}^{\infty}g_k$ and $\omega(g_k,\delta) \leq \omega(\delta)$, then $\omega(g,\delta) \lesssim \omega(\delta)$. 
\end{lemma}
Although the next lemma is presumably already known, the authors do not know of a reference, so a proof is included.
\begin{lemma}\label{integral}
Let $a,b \in \mathbb{N}$ and $a<b$. Then
\begin{equation*}
\int_{\frac{a\pi}{n}}^{\frac{b\pi}{n}} \frac{\sin^{2} nt}{t} dt \geq \frac{1}{12} \sum\limits_{i=a}^{b}\frac{1}{i}.
\end{equation*}
\end{lemma}
\begin{proof}
{\bf Case I.} Suppose that $2a< b$. Then we have $\ln{\frac{b}{a}} > \ln{2}> 0.32$ and
\begin{equation}\label{inn}
\frac{1}{2}\Big(\ln{\frac{b}{a}}-\frac{1}{10}\Big)>\frac{1}{12}\Big(1+\ln{\frac{b}{a}}\Big).
\end{equation}
Also since
$$
\int_{\pi}^{\infty} \frac{\cos 2s}{s} ds<\frac{1}{10},
$$
and 
\begin{align*}\label{simpeqle}
\int_{\frac{a\pi}{n}}^{\frac{b\pi}{n}} \frac{\sin^{2} nt}{t} dt &= \frac{1}{2} \Big( \int_{\frac{a\pi}{n}}^{\frac{b\pi}{n}} \frac{dt}{t}-\int_{\frac{a\pi}{n}}^{\frac{b\pi}{n}} \frac{\cos 2nt}{t} dt\Big)\\[.1in]
&=\frac{1}{2} \Big( \int_{\frac{a\pi}{n}}^{\frac{b\pi}{n}} \frac{dt}{t}-\int_{a\pi}^{b\pi} \frac{\cos 2s}{s} ds\Big),
\end{align*}
by (\ref{inn}) we get
$$
\int_{\frac{a\pi}{n}}^{\frac{2a\pi}{n}} \frac{\sin^{2} nt}{t} dt 
>\frac{1}{2}\Big(\ln{\frac{b}{a}}-\frac{1}{10}\Big)>\frac{1}{12}\Big(1+\ln{\frac{b}{a}}\Big)\geq \frac{1}{12} \sum_{i=a}^{b}\frac{1}{i}.
$$

\medskip
{\bf Case II.} Let $2a \geq b$. Denote $c=b-a\geq 1$. Observe that $3a>c$ and $ac+a+c \leq 3ac$. Therefore 
$$
ac+a+c+c^2\leq 6ac.
$$
Evidently, 
$$
\frac{c+1}{a}\leq \frac{6c}{a+c}
$$
and
$$
\sum\limits_{i=a}^{b}\frac{1}{i} \leq \frac{b-a+1}{a}\leq \frac{6(b-a)}{b}.
$$
Thus we obtain 
$$
\frac{b-a}{2b} \geq  \frac{1}{12} \sum\limits_{i=a}^{b}\frac{1}{i}. 
$$
Now, using the above inequality, it follows that
$$
\int_{\frac{a\pi}{n}}^{\frac{b\pi}{n}} \frac{\sin^{2} nt}{t} dt \geq \frac{n}{b \pi} \int_{\frac{a\pi}{n}}^{\frac{b\pi}{n}} \sin^{2} nt dt=\frac{n}{b \pi}\cdot \frac{1}{2} \Big(\int_{\frac{a\pi}{n}}^{\frac{b\pi}{n}} dt -\int_{\frac{a\pi}{n}}^{\frac{b\pi}{n}} \cos 2nt dt\Big)
$$
$$
=\frac{n}{b \pi}\cdot\frac{(b-a)\pi}{2n} \geq  \frac{1}{12} \sum\limits_{i=a}^{b}\frac{1}{i}.
$$
\end{proof}

The terms $Q_k$ defined in the next lemma naturally appear in proving the equivalence of (ii), (iii), (iv) and (v) in Theorem \ref{convergence} (see next subsection). Although the proof is elementary, we include the details for completeness.

\begin{lemma}\label{Q}
Let $1\leq p <\infty$. Then for each positive integer $k$ we have
$$
1-p^{-1}\leq Q_k\leq 2^{-\frac1{p}}, \ \ \text{where} \ \ Q_k:=1-k+\frac{k^{1+\frac1{p}}}{(k+1)^\frac1{p}}.
$$
\end{lemma}

\begin{proof}
We first show that the $Q_k$ are bounded by $2^{-\frac1{p}}$. To this end, we define 
$$
Q(t):=1-t+\frac{t^{1+\frac1{p}}}{(t+1)^\frac1{p}},\quad (t\geq 1),
$$
so that $Q(k)=Q_k$, $k\geq 1$. One can verify that
$$
Q'(t)=R(t)-1, \ \ \ \text{where} \ \ \ R(t)=\left(\frac{t}{t+1}\right)^{\frac1{p}}\left(1+\frac1{p(t+1)}\right),
$$
and 
$$
R'(t)=\frac1{p}\left(1+\frac1{p}\right)\left(\frac{t^{\frac1{p}-1}}{(t+1)^{2+\frac1{p}}}\right)>0, \ \ \ t\geq 1.
$$
Since $R(t)$ is increasing and $R(t)\rightarrow 1$ as $t\rightarrow \infty$, we conclude that $R(t)\leq 1$ and hence $Q'(t)\leq 0$, which in turn implies that $Q(t)$ is decreasing. As a consequence, $Q(t)$ attains its maximum at $t=1$, thus $Q_k\leq Q_1=2^{-\frac1{p}}$ for all $k$. On the other hand, since $Q(t)\rightarrow 1-\frac1{p}$ as $t\rightarrow\infty$, we infer that $Q_k\geq 1-\frac1{p}$ for all $k$.
\end{proof}

\medskip
\subsection{Proof of Theorem \ref{convergence}}

We begin with the difficult part whose proof amounts to constructing a function $f_0$ in $H^\omega\cap V_p[\nu]$ that is closely connected to the sequence $\sigma$.

{\bf $\text{(i)}\Rightarrow\text{(ii)}$.} To prove this part, we first define a sequence $\{f_k\}$ of $2\pi$-periodic functions as follows:
$$ f_k(x):=\begin{cases}
 \sin(l_k x)\omega\Big(\frac{1}{l_k}\Big),  &\textmd{if}\ x\in \Big[\frac{\psi(l_k)}{l_k}\pi,\frac{\mu(l_k)}{l_k}\pi\Big];  \\ \\
\sin(l_k x)\chi(l_k)\varepsilon_p(r), &\textmd{if}\ x\in \Big[\frac{r}{l_k}\pi,\frac{(r+1)}{l_k}\pi\Big];\ r=\mu(l_k)+1,\dots, M(l_k);\\ \\
0, &\textmd{otherwise in}~  [0,2\pi].
\end{cases}
$$
Then, since $\varphi(n)=o(n)$ by Lemma \ref{Lemma 4.5}(a), the $f_k$ have disjoint supports and so the function $f_0:=\sum\limits_{k=1}^{\infty} f_k$
is well-defined. Now, it is enough to show that $f_0\in H^\omega\cap V_p[\nu]$ while there exist a subsequence $\{m_i\}$ of $\{l_k\}\subseteq\{n_i\}$ and a positive integer $N$ such that $\sigma(m_i)\lesssim|f_0(0)-S_{m_i}(f_0,0)|$ for $i\geq N$ ($S_n(x):=S_n(f_0,x)$ designates the $n$th partial sum of the Fourier series of $f_0$). Towards this end, we need to verify the claims made below.

\medskip
\noindent{\bf Claim I:} $f\in H^\omega$.

\medskip
Let $\theta(l_k) \leq \varphi(l_k)$. If $x\in \Big[\frac{\psi(l_k)}{l_k}\pi,\frac{\theta(l_k)}{l_k}\pi\Big]$, then  $\omega(f_k,\delta) \leq \omega(\delta)$. Now let $x\in \Big[\frac{r\pi}{l_k},\frac{(r+1)\pi}{l_k}\Big]$ and $r=\theta(l_k),\dotsb ,\varphi(l_k)-1$. Applying Lemma \ref{thetaprop} for $r=\theta(l_k)+1,\dotsb ,\varphi(l_k)-1$ along with the fact that $\varepsilon_p(r)\downarrow$, we get
\begin{equation}\label{estnu1}
 \varepsilon_p(r)\leq \varepsilon_p(\theta(l_k)+1)\leq \frac{\nu(\theta(l_k)+1)}{(\theta(l_k)+1)^\frac1{p}}\leq \omega\Big(\frac{1}{l_k}\Big).
\end{equation}
Thus  $\omega(f_k,\delta) \leq 2 \omega(\delta)$ for $0 \leq \delta \leq \pi$. By Lemma \ref{Oskol} and similar arguments as in \cite{Osk}, we observe that $f_0\in H^\omega$. Likewise, one can verify that the above argument is valid when $\theta(l_k) > \varphi(l_k)$.

\medskip
\noindent{\bf Claim II:} $f\in V_p[\nu]$.

\medskip
We prove the claim for the case $\theta(l_k)\leq \phi(l_k)$; the proof of the case $\theta(l_k)> \phi(l_k)$ is similar. Evidently by definition of $\varphi$ and $\psi$ we have
\begin{equation}\label{distance}
\varphi(l_k)-\psi(l_k) > \frac{2}{\omega\big(\frac{1}{l_k}\big)}-\frac{1}{\omega\big(\frac{1}{l_k}\big)}=\frac{1}{\omega\big(\frac{1}{l_k}\big)}\rightarrow\infty.
\end{equation}
Thus $\sum\limits_{k=1}^{\infty} (\varphi(l_k)-\psi(l_k))=\infty$.
Applying \eqref{distance} and Lemma \ref{Lemma 4.5}(c), we obtain
\begin{align}\label{phiminpsi}
\varphi(l_k)-\psi(l_k)\nonumber&<\psi(l_{k+1})-\psi(l_k)\\[.1in]
\nonumber&<\psi(l_{k+1})<\omega\left(\frac{1}{l_{k+1}}\right)^{-1}\\[.1in]
&<\varphi(l_{k+1})-\psi(l_{k+1}).
\end{align}

Applying (\ref{estnu1}) and Lemma \ref{Lemma 4.5}(b) yields 
 \begin{align}\label{infsupfkfk+1}
  \omega\left(\frac{1}{l_k}\right) \nonumber&\
  \geq \varepsilon_p(\theta(l_k)+1)\geq \varepsilon_p(\phi(l_k))\\[.1in]
  &\geq \omega\left(\frac{1}{l_{k+1}}\right) \geq \varepsilon_p(\theta(l_{k+1})+1).
 \end{align}
At the same time, for all $m\in \mathbb{N}$, we have
\begin{equation}\label{nupp-rel1}
\upsilon_p(\varphi(l_k)-\psi(l_k)+m,f_k)=\upsilon_p(\varphi(l_k)-\psi(l_k),f_k).
\end{equation}
Also, if $m< \varphi(l_k)-\psi(l_k)$, then by \eqref{phiminpsi} and  \eqref{infsupfkfk+1} we get 
\begin{equation}\label{nupp-rel2}
\upsilon_p(m,f_k) \geq \upsilon_p(m,f_{k+1}).
\end{equation}

Now we estimate the modulus of variation of $f_0$. Let $n$ be large enough. There exists $k_0$ such that
$$
n_0:=\sum_{k=1}^{k_0-1} (\varphi(l_k)-\psi(l_k))\leq n <\sum_{k=1}^{k_0} (\varphi(l_k)-\psi(l_k)).
$$
Using \eqref{nupp-rel1}, for $1\leq k\leq k_0-1$, we obtain
\begin{equation}\label{nupp-rel3}
\upsilon_p(\varphi(l_k)-\psi(l_k)+n-n_0,f_k)=\upsilon_p(\varphi(l_k)-\psi(l_k),f_k).
\end{equation}
On the other hand, it follows from \eqref{phiminpsi} that 
\begin{equation}\label{uppern0}
n-n_0 < \varphi(l_{k_0+1})-\psi(l_{k_0+1}).
\end{equation}
In view of \eqref{nupp-rel2} and \eqref{uppern0}, we conclude that
\begin{equation}\label{nupp-rel4}
\upsilon_p(n-n_0,f_k) \geq \upsilon_p(n-n_0,f_{k+1}), \ \ \ k\geq k_0+1.
\end{equation}
Hence, due to \eqref{nupp-rel1}, \eqref{nupp-rel3} and \eqref{nupp-rel4}, it follows that
\begin{align}\label{Cha31}
\upsilon_p(n,f_0) & \nonumber=\upsilon_p\Big(n-n_0+\sum_{k=1}^{k_0-1} (\varphi(l_k)-\psi(l_k)),f_0\Big)\\[.1in]
&\leq \sum_{k=1}^{k_0-1} \upsilon_p\big(\varphi(l_k)-\psi(l_k),f_k\big) + \upsilon_p(n-n_0,f_{k_0}).
\end{align}
 
To simplify notation, we omit $l_k$ in $\theta(l_k)$, $\psi(l_k)$ and $\varphi(l_k)$ which should not cause any ambiguity. Suppose $\theta < \varphi$. 
We get 
\begin{align*}
\upsilon_p \Big(\varphi-\psi,f_k \Big) &\leq 2\omega\left(\frac{1}{l_k}\right)\Big( \theta-\psi \Big)^{\frac{1}{p}} + 2 \Big(  \sum_{r=\theta+1}^{\varphi}\varepsilon_p(r)^p \Big)^{\frac{1}{p}}\\[.1in]
&=2\omega\left(\frac{1}{l_k}\right) \Big(\theta-\psi\Big)^{\frac{1}{p}} + 2\Big(\nu(\varphi)^p-\nu(\theta)^p\Big)^{\frac{1}{p}}\\[.1in]
&\leq 2\omega\left(\frac{1}{l_k}\right) \theta^{\frac{1}{p}}+2\nu(\varphi)\leq 4\nu(\varphi),
\end{align*}
where the last inequality follows from Lemma \ref{thetaprop}.

It follows from Lemma \ref{definitionphipsi}(c)  that
$$
4 \nu(\psi)<\nu(\theta)< \nu(\varphi)
$$
or
$$
\nu(\varphi)<2(\nu(\varphi)-\nu(\psi)).
$$

Hence
\begin{equation}\label{Cha33}
\upsilon_p \Big(\varphi(l_k)-\psi(l_k) , f_k    \Big) \leq 8 \Big(\nu(\varphi(l_k))-\nu(\psi(l_k))  \Big).
\end{equation}
A similar argument is valid when $\theta > \varphi$.\\ 

\noindent {\bf Estimation of $\upsilon_p(n-n_0,f_{k_0})$:} We deal with the following three cases.\\

\underline{Case 1}:  $\theta(l_{k_0})< \varphi(l_{k_0})$ and $n-n_0 \leq \theta(l_{k_0})-\psi(l_{k_0})$;\\

\underline{Case 2}:  $\theta(l_{k_0})< \varphi(l_{k_0})$ and  $n-n_0 > \theta(l_{k_0})-\psi(l_{k_0})$;\\

\underline{Case 3}:  $\theta(l_{k_0}) \geq \varphi(l_{k_0})$.\\

\noindent As Lemma \ref{Lemma 4.5}(c) readily yields $n_0<\psi(l_{k_0})$, in Case 1 we have
\begin{equation}\label{nlessm0}
n \leq \theta(l_{k_0})-\psi(l_{k_0})+n_0< \theta(l_{k_0}).
\end{equation}
\noindent By abuse of notation, we write $\varphi:=\varphi(l_{k_0})$, $\theta:=\theta(l_{k_0})$ and $\psi:=\psi(l_{k_0})$ for the sake of simplicity. Now, by Lemma \ref{thetaprop} we obtain
\begin{align*}
\upsilon_p(n-n_0,f_{k_0}) &\leq  \Big( \sum_{r=1}^{n-n_0} \Big(2\omega \Big(\frac{1}{l_{k_0}}\Big)\Big)^p  \Big)^{\frac{1}{p}}\\[.1in]
&\leq  2\omega \Big(\frac{1}{l_{k_0}}\Big)\Big( n-n_0  \Big)^{\frac{1}{p}}\\[.1in]
&\leq 2\omega \Big(\frac{1}{l_{k_0}}\Big) n^{\frac{1}{p}}\leq 2 \frac{\nu(\theta)}{\theta^{\frac{1}{p}}}  n^{\frac{1}{p}}.
\end{align*}
Also, by \eqref{nlessm0} and taking into account the fact that $\frac{\nu(r)}{r^{\frac1{p}}}\downarrow$, we get
\begin{equation}\label{Cha34}
\upsilon_p(n-n_0,f_{k_0}) \leq 2 \frac{\nu(\theta)}{\theta^{\frac{1}{p}}}  n^{\frac{1}{p}}\leq 2 \frac{\nu(n)}{n^{\frac{1}{p}}}  n^{\frac{1}{p}}=2\nu(n).
\end{equation}

Let us now consider Case 2, that is, when $\theta< \varphi$ and  $n-n_0 > \theta-\psi$. First note that by Lemma \ref{definitionphipsi}(c), we have $\psi^2 < \theta$. As a result,
\begin{equation}\label{psi<n}
n > \theta-\psi+n_0 >  \psi^2-\psi+n_0> \psi.
\end{equation}
Recalling the definition of $f_{k_0}$, we observe that

\begin{align*}
\Big(\upsilon_p(n-n_0,f_{k_0})\Big)^p &\leq 2^p \omega\Big(\frac{1}{l_{k_0}}\Big)^p (\theta-\psi)+2^p \sum_{r=\theta+1}^{n-n_0+\psi}\varepsilon_p(r)^p\\[.1in]
&\hspace{0.1cm}= 2^p\omega\Big(\frac{1}{l_{k_0}}\Big)^p (\theta-\psi)+ 2^p \Big(\nu(n-n_0+\psi)^p-\nu(\theta)^p\Big)\\[.1in]
&\hspace{-.2cm}\overset{\text{(Lemma \ref{thetaprop})}}{\leq}  2^p \frac{\nu(\theta)^p}{\theta} \theta + 2^p\Big(\nu(n-n_0+\psi)^p-\nu(\theta)^p\Big)\\[.1in]
&\hspace{.3cm}= 2^p\nu(n-n_0+\psi)^p.
\end{align*}
Thus, by (\ref{psi<n}) one obtains
\begin{align}\label{Cha35}
\upsilon_p(n-n_0,f_{k_0})\nonumber&\leq 2\nu(n-n_0+\psi)\\[.1in]
\nonumber&\leq 2\nu(n)+2\nu(\psi)\\[.1in]
&\leq 4\nu(n).
\end{align}

Now we turn to Case 3. In this case, due to the choice of $k_0$, we see that 
\begin{align*}
n-n_0 &< \varphi(l_{k_0})-\psi(l_{k_0})\\[.1in]
&\leq \theta(l_{k_0})-\psi(l_{k_0})
\end{align*}
With an argument similar to that of Case 1, it follows  that
\begin{equation}\label{Cha35+1}
\upsilon_p(n-n_0,f_{k_0}) <2\nu(n).
\end{equation}
Applying (\ref{Cha33}), (\ref{Cha34}), (\ref{Cha35}) and (\ref{Cha35+1}), we obtain from (\ref{Cha31}) that
$$
\upsilon_p(n,f_0) \lesssim \nu(n)+\sum_{k=1}^{k_0-1} \nu(\varphi(l_k)-\psi(l_k)).
$$
But note that
$$
 \aligned\sum_{k=1}^{k_0-1} \nu(\varphi(l_k))-\nu(\psi(l_k)) &\leq \sum_{k=1}^{k_0-1} \nu(\varphi(l_k))-\nu(\varphi(l_{k-1}))\\
&\leq\nu(\varphi(l_{k_0-1}))\leq \nu( \varphi(l_{k_0-1})-\psi(l_1))+ \nu(\psi(l_1)).\endaligned
$$
Since $\varphi(n)\geq 2 \psi(n)$ for all $n$, we have

$$
\psi (l_{k_0-1})\leq \varphi(l_{k_0-1}) - \psi (l_{k_0-1}).
$$

As a result, we get 
$$
\psi(l_1)\leq\psi (l_{k_0-1}) \leq \sum_{k=1}^{k_0-1}\varphi(l_k)-\psi(l_k) <n,
$$
and thus

$$
\aligned\nu \Big( \varphi(l_{k_0-1})-\psi(l_1) \Big) &\leq  \nu \Big( \varphi(l_{k_0-1})-\psi(l_1)+ \sum_{k=2}^{k_0-2} \varphi(l_k)-\psi(l_k) \Big)
\\
&\leq\nu \Big( \psi (l_{k_0-1})+\sum_{k=1}^{k_0-1} \varphi(l_k)-\psi(l_k) \Big)\\
&\leq \nu \Big(2  \sum_{k=1}^{k_0-1} \varphi(l_k)-\psi(l_k) \Big)\\ 
&\leq 2 \nu(n).\endaligned
$$
Consequently,
$$
\sum_{k=1}^{k_0-1} \nu(\varphi(l_k))-\nu(\psi(l_k))\leq 3\nu(n).
$$
Therefore $f_0 \in V_p[\nu]$.

\medskip
\noindent{\bf Claim III:} There exist a subsequence $\{m_i\}$ of $\{l_k\}\subseteq\{n_i\}$ and a positive integer $N$ such that $\sigma(m_i)\lesssim|f_0(0)-S_{m_i}(f_0,0)|$ for $i\geq N$.

\medskip
First of all, note that if the claim holds true, then since by hypothesis we have $|f_0(0)-S_{m_i}(f_0,0)|\rightarrow0$, we conclude that $\sigma(m_i)\rightarrow0$ as $i\rightarrow\infty$. But the convergence of $\{\sigma(n_i)\}$ implies
$$
\limsup_{n\rightarrow\infty}\sigma(n)=\lim_{i\rightarrow\infty} \sigma(n_i)=\lim_{i\rightarrow\infty} \sigma(m_i)=0.
$$
This means that $\sigma(n)\rightarrow0$ as $n\rightarrow\infty$.

\medskip
In order to see the claim, it suffices to show that
\begin{equation}\label{disconvergence}
\limsup_{k \to \infty} \frac{\left|f_0(0)-S_{l_k}(f_0,0)\right|}{\sigma(l_k)}=\alpha>0,
\end{equation}
since in that case, we may find a subsequence $\{m_i\}$ of $\{l_k\}$ such that
$$
\lim_{i \to \infty} \frac{\left|f_0(0)-S_{m_i}(f_0,0)\right|}{\sigma(m_i)}=\alpha.
$$
This implies the existence of a positive integer $N$ such that
$$
\frac{\alpha}{2}\cdot\sigma(m_i)<\left|f_0(0)-S_{m_i}(f_0,0)\right|, \ \ \ \text{for} \ i\geq N.
$$

Let us now prove \eqref{disconvergence}. It is well known (see, e.g.,  \cite[p. 101]{Bar}) that for any $2\pi$-periodic continuous function $f(x)$, we have the estimate 
 \begin{equation}\label{4.26}
   S_n(f,x)- \widetilde{S}_n(f,x)= O\Big(\omega\Big(\frac{\pi}{n},f\Big)\Big),
 \end{equation}
 where 
 $$
 \widetilde{S}_n(f,x)=\frac{1}{\pi} \int_{-\pi}^{\pi}f(x+t) \frac{\sin nt}{t} dt.
 $$
Hence, taking into account the fact that $f_0(0)=0$, we get 
 \begin{equation*}\label{modbari}
\left|S_{l_k}(f_0,0)-f_0(0) \right|> \left|\widetilde{S}_{l_k}(f_0,0)-f_0(0) \right|- C\cdot \omega\Big(\frac{1}{l_k},f_0\Big),
\end{equation*}
where $C$ is a positive constant. Applying Lemma \ref{integral} with $n:=l_k$, $a:=\psi(l_k)$ and $b:=\theta(l_k)$ yields

$$
\omega\Big(\frac{1}{l_k}\Big) \int_{\frac{\psi(l_k) \pi}{l_k}}^{\frac{\theta (l_k)\pi}{l_k}} \frac{\sin^{2} l_kt}{t} dt \geq \frac{1}{12} \omega\Big(\frac{1}{l_k}\Big)  \sum_{i=\psi(l_k)}^{\theta(l_k)}\frac{1}{i}.
$$
Also, since
$$
\int_{\frac{r\pi}{l_k}}^{\frac{(r+1)\pi}{l_k}} \frac{\sin^{2} l_kt}{t} dt > \frac{1}{12r},
$$
we obtain
$$
\aligned\widetilde{S}_{l_k}(f_k,0)&=\frac{1}{\pi} \int_{0}^{2\pi}f_k(t) \frac{\sin l_kt}{t} dt\\
&=\frac{1}{\pi} \left\{\omega\left(\frac1{l_k}\right)\int_{\frac{\psi(l_k) \pi}{l_k}}^{\frac{\theta (l_k)\pi}{l_k}}\frac{\sin^{2} l_kt}{t} dt+\sum_{r=\theta(l_k)+1}^{\phi(l_k)} \varepsilon_p(r)\int_{\frac{r\pi}{l_k}}^{\frac{(r+1)\pi}{l_k}}\frac{\sin^{2} l_kt}{t} dt\right\}\\
&\geq \frac{1}{12\pi}\left\{\omega\left(\frac1{l_k}\right)\sum_{i=\psi(l_k)}^{\theta(l_k)}\frac{1}{i}+\sum_{r=\theta(l_k)+1}^{\phi(l_k)} \frac{\varepsilon_p(r)}{r}\right\}=\frac{1}{12\pi}\bar{\sigma}(l_k).\endaligned
$$

Now define
$$ 
G_k(x):=\sum_{s=1}^{k-1} f_s(x)
$$
and 
$$
F_k(x):=\sum_{s=k+1}^{\infty} f_s(x).
$$
The rest of the proof follows along the lines of the last part of the proof of \cite[Theorem 3, p. 492]{Cha2}. So we omit the details and we only give a sketch of proof.

As $G_k$ is of bounded variation, by making use of an estimate of Stechkin (see \cite[Inequality (7)]{Osk2}) along with \eqref{4.26} and Lemma \ref{Lemma 4.5}(e) (taking $\bar{c}=1/36\pi\tilde{c}$), we get
\begin{align*}
\widetilde{S}_{l_k}(G_k,0)&\leq \tilde{c} A_k\ln\left(A_k^{-1}\sum_{i=1}^{k-1}\omega\left(\frac{1}{l_i}\right)\varphi(l_i)\right)\\[.1in]
&<\bar{c} \ \bar{\sigma}(l_k)<\frac{1}{36\pi}\bar{\sigma}(l_k).
\end{align*}

In addition, since $\Big(0,\frac{\psi(l_k)\pi}{l_k}\Big)$ contains the support of $F_k$, Lemma \ref{Lemma 4.5}(d) implies
$$
 \widetilde{S}_{l_k}(F_k,0)<\frac{1}{36\pi}\bar{\sigma}(l_k),
$$
Thus we arrive at
$$
\left|\widetilde{S}_{l_k}(f_0,0)\right|\geq \left| \widetilde{S}_{l_k}(f_k,0) \right|-\left| \widetilde{S}_{l_k}(F_k,0) \right|-\left| \widetilde{S}_{l_k}(G_k,0) \right| \geq \frac{1}{36\pi}\bar{\sigma}(l_k).
$$
By applying Lemma \ref{lemma 4.4}, we obtain
$$
\frac{\left|S_{l_k}(f_0,0)-f_0(0) \right|}{\sigma(l_k)}>C_1 - C_2\frac{\omega\Big(\frac{1}{l_k},f_0\Big)}{\sigma(l_k)},
$$
while $f_0\in H^{\omega}$ implies
$$
\frac{\left|S_{l_k}(f_0,0)-f_0(0) \right|}{\sigma(l_k)}>C_3 - C_4\frac{\omega\Big(\frac{1}{l_k}\Big)}{\sigma(l_k)},
$$
where $C_i>0$ ($i=1,\cdots,4$) are positive constants. Note that by Lemma \ref{lemma 4.4}(b), it follows that $\omega(1/l_k)/\sigma(l_k)\rightarrow0$ as $k\rightarrow\infty$, and hence
$$
\limsup_{k\rightarrow\infty} \frac{\left|S_{l_k}(f_0,0)-f_0(0) \right|}{\sigma(l_k)}>0.
$$
This completes the proof. 

\medskip
\medskip
\medskip
{\bf $\text{(ii)}\Rightarrow\text{(i)}$.}
Denote
\begin{equation}\label{2etI}
I:= \sum_{k=1}^{n} \frac{(-1)^k(2n+1)}{2 \pi^2 k} \int_{0}^{\frac{\pi}{2n+1}} \sin\Big(\frac{2n+1}{2}t\Big) \left\{\sum_{r=0}^1 f(\xi_{k,r}+t)-f(\xi_{k,r}-t)\right\} dt,
\end{equation}
where
$$
\xi_{k,r}:=x+(-1)^r \frac{2k\pi}{2n+1}. 
$$
In \cite{Nik}, Nikol'skii found the estimate
\begin{equation*}\label{2etn}
f(x)-S_n(f,x)\lesssim I + \omega\Big(\frac{1}{n},f\Big).
\end{equation*}

It is well-known that for all $n \in  \mathbb{N}$, using subadditivity of the modulus of continuity $\omega(t)$, we have $\omega(nt) \leq n \omega(t)$. Thus for $t \in \big[0,\frac{\pi}{2n+1}\big]$, one has
\begin{equation}\label{2et22}
\Big| f(\xi_{k,r}+t)-f(\xi_{k,r}-t) \Big|  \lesssim \omega\Big(\frac{1}{n},f\Big), \ \ \ r=0,1. 
\end{equation}
Further, suppose $\left|\sum_{r=0}^1 f(\xi_{k,r}+t)-f(\xi_{k,r}-t)\right|$ attains its maximum on $\big[0,\frac{\pi}{2n+1}\big]$ at $t_k$. With this in mind, (\ref{2etI}) yields

\begin{equation}\label{2et24}
|I|  \lesssim \sum_{k=1}^{n} \frac{1}{k} \left|\sum_{r=0}^1 f(\xi_{k,r}+t_k)-f(\xi_{k,r}-t_k)\right|.
\end{equation}
By (\ref{2et22}), and using Abel's transformation \eqref{Abel} with
$$
x_k:=\frac1{k} \ \ \ \text{and} \ \ \ y_k:=\Big| f(\xi_{k,r}+t_k)-f(\xi_{k,r}-t_k) \Big|,
$$
we get
\begin{align*}\label{2et25}
|I|  \lesssim \sum_{k=1}^{\theta} \frac{1}{k}\omega\Big(\frac{1}{n},f\Big) +\sum_{r=0}^1 \sum_{k=\theta+1}^{n-1} \frac{1}{k^2} \sum_{j=1}^k \Big| f(\xi_{j,r}+t_j)-f(\xi_{j,r}-t_j) \Big|\\[.1in]
&\hspace{-7.6cm}+\sum_{r=0}^1 \frac{1}{n} \sum_{j=1}^n \Big| f(\xi_{j,r}+t_j)-f(\xi_{j,r}-t_j) \Big|.
\end{align*}
Now an application of H\"{o}lder's inequality yields
\begin{align*}
|I|  \lesssim \sum_{k=1}^{\theta} \frac{1}{k}\omega\Big(\frac{1}{n},f\Big) +\sum_{r=0}^1 \sum_{k=\theta+1}^{n-1} \frac{1}{k^{1+\frac1{p}}} \left\{\sum_{j=1}^k \Big| f(\xi_{j,r}+t_j)-f(\xi_{j,r}-t_j) \Big|^p\right\}^{\frac1{p}}\\[.1in]
&\hspace{-9.1cm}+\sum_{r=0}^1 \frac{1}{n^{\frac1{p}}} \left\{\sum_{j=1}^n \Big| f(\xi_{j,r}+t_j)-f(\xi_{j,r}-t_j) \Big|^p\right\}^{\frac1{p}}.
\end{align*}

Since for $t_k \in \big[0, \frac{\pi}{n+1}\big]$, $k=1,\cdots,n$ ($k\neq \theta$),
$$
\Big(\xi_{k,r}-t_k,\xi_{k,r}+t_k\Big) \bigcap \Big(\xi_{\theta,r}-t_\theta,\xi_{\theta,r}+t_\theta\Big)= \emptyset,
$$
it follows from the definition of the modulus of $p$-variation that
\begin{equation*}\label{2et27}
\sum_{k=1}^{\theta} \Big|f(\xi_{k,r}+t_k)-f(\xi_{k,r}-t_k)\Big|^p \leq (\upsilon_p(\theta,f))^p.
\end{equation*}

As a result, if $f\in H^\omega\cap V_p[\nu]$, applying Abel's transformation and Lemma \ref{epsionprop}(iii) yields
\begin{align*}
\|f(x)-S_n(f,x)\|_{C(0,2\pi)} &\lesssim \omega\Big(\frac{1}{n},f\Big) \sum_{k=1}^{\theta} \frac{1}{k} + \sum_{k=\theta+1}^{n-1}\frac{\upsilon_p(k)}{k^{1+\frac{1}{p}}}+\frac{\upsilon_p(n)}{n^\frac{1}{p}}\\[.1in]
&\lesssim \omega\Big(\frac{1}{n}\Big) \sum_{k=1}^{\theta} \frac{1}{k} + \sum_{k=\theta+1}^{n-1}\frac{\nu(k)-\nu(k-1)}{k^{\frac{1}{p}}}+\frac{\nu(n)}{n^\frac{1}{p}}\\[.1in]
&\lesssim \omega\Big(\frac{1}{n}\Big) \sum_{k=1}^{\theta} \frac{1}{k} + \sum_{k=\theta+1}^{n-1}\frac{\varepsilon_p(k)}{k}+\frac{\nu(n)}{n^\frac{1}{p}}=\sigma(n)+\frac{\nu(n)}{n^\frac{1}{p}}.
\end{align*}
Since $\nu(n)=o(n^\frac{1}{p})$, we conclude that $\|f(x)-S_n(f,x)\|_{C(0,2\pi)}\rightarrow 0$ as $n\rightarrow 0$.

{\bf $\text{(ii)}\Leftrightarrow\text{(iii)}\Leftrightarrow\text{(iv)}\Leftrightarrow\text{(v)}$.}
By applying Abel's transformation we obtain 
\begin{align*}
\sum_{k=\theta+1}^{n-1}\frac{\nu(k)-\nu(k-1)}{k^{\frac{1}{p}}}&=\sum_{k=\theta+1}^{n-2}\Delta\big(k^{-\frac1{p}}\big) \sum_{j=\theta+1}^{k}\nu(j)-\nu(j-1)+\sum_{j=\theta+1}^{n-1}\frac{\nu(j)-\nu(j-1)}{(n-1)^{\frac{1}{p}}}\\
&= \sum_{k=\theta+1}^{n-2}\Delta\big(k^{-\frac1{p}}\big)(\nu(k)-\nu(\theta)) + \frac{\nu(n-1)-\nu(\theta)}{(n-1)^{\frac{1}{p}}}\\[.1in]
&= \sum_{k=\theta+1}^{n-2}\Delta\big(k^{-\frac1{p}}\big)\nu(k)- \sum_{k=\theta+1}^{n-2}\Delta\big(k^{-\frac1{p}}\big)\nu(\theta)+\frac{\nu(n-1)-\nu(\theta)}{(n-1)^{\frac{1}{p}}}\\[.1in]
&=\sum_{k=\theta+1}^{n-2}\Delta\big(k^{-\frac1{p}}\big)\nu(k)- \frac{\nu(\theta)}{(\theta+1)^{\frac{1}{p}}} + \frac{\nu(\theta)}{(n-1)^{\frac{1}{p}}}+\frac{\nu(n-1)-\nu(\theta)}{(n-1)^{\frac{1}{p}}}\\[.1in]
& =\sum_{k=\theta+1}^{n-1}\Delta\big(k^{-\frac1{p}}\big)\nu(k)+\frac{\nu(n-1)}{n^{1/p}} - \frac{\nu(\theta)}{(\theta+1)^{\frac{1}{p}}}.
\end{align*}
This implies $\text{(iv)}\Leftrightarrow\text{(v)}$. Furthermore,
\begin{equation}\label{equi1}
\sum_{k=\theta+1}^{n-1}\frac{\nu(k)}{k^{1+\frac{1}{p}}}=\frac{\nu(\theta)}{(\theta+1)^{\frac{1}{p}}}-\frac{\nu(n-1)}{n^{1/p}}+\sum_{k=\theta+1}^{n-1}\frac{\nu(k)-\nu(k-1)}{k^{\frac{1}{p}}}+\sum_{k=\theta+1}^{n-1}\frac{\nu(k)}{ k^{1+\frac{1}{p}}}Q_k,
\end{equation}
where $Q_k$ is defined as in Lemma \ref{Q}, using which we obtain
\begin{equation}\label{milad}
(1-p^{-1})\sum_{k=\theta+1}^{n-1}\frac{\nu(k)}{ k^{1+\frac{1}{p}}}<\sum_{k=\theta+1}^{n-1}\frac{\nu(k)}{ k^{1+\frac{1}{p}}}Q_k< 2^{-\frac1{p}}\sum_{k=\theta+1}^{n-1}\frac{\nu(k)}{k^{1+\frac{1}{p}}}.
\end{equation}
Combining these inequalities with \eqref{equi1}, we see that $\text{(iii)}\Leftrightarrow\text{(iv)}$.

To conclude, let us verify $\text{(ii)}\Leftrightarrow\text{(iii)}$. Applying Abel's transformation and H\"{o}lder's inequality yields
\begin{align}\label{ccc}
\nonumber\sum_{k=\theta+1}^{n-1} \frac{\varepsilon_p(k)}{k}&=\sum_{k=\theta+1}^{n-2}\frac1{k(k+1)}\sum_{j=\theta+1}^k \varepsilon_p(j)+\frac1{n-1}\sum_{j=\theta+1}^{n-1} \varepsilon_p(j)\\[.1in]
\nonumber&\leq\sum_{k=\theta+1}^{n-2}\frac1{k(k+1)}\Big(\sum_{j=\theta+1}^k \varepsilon_p(j)^p\Big)^{\frac1{p}}k^{1-\frac1{p}}+\frac1{n-1}\Big(\sum_{j=\theta+1}^{n-1} \varepsilon_p(j)^p\Big)^{\frac1{p}}(n-1)^{1-\frac1{p}}\\[.1in]
\nonumber&=\frac{\big(\nu(n-1)^p-\nu(\theta)^p\big)^{\frac1{p}}}{(n-1)^{\frac1{p}}}+\sum_{k=\theta+1}^{n-2}\frac{\big(\nu(k)^p-\nu(\theta)^p\big)^{\frac1{p}}}{k^{\frac1{p}}(k+1)}\\[.1in]
&\leq\frac{\nu(n-1)}{(n-1)^{\frac1{p}}}+\sum_{k=\theta+1}^{n-1}\frac{\nu(k)}{k^{1+\frac1{p}}}.
\end{align}
Also, using \eqref{equi1}, \eqref{milad} and Lemma \ref{epsionprop}(iii) we get
\begin{align}\label{cdd}
\sum_{k=\theta+1}^{n-1} \frac{\nu(k)}{k^{1+\frac1{p}}}\nonumber&\lesssim \frac{\nu(\theta)}{(\theta+1)^\frac1{p}}-\frac{\nu(n-1)}{n^\frac1{p}}+\sum_{k=\theta+1}^{n-1}\frac{\nu(k)-\nu(k-1)}{k^\frac1{p}}\\[.1in]
&\lesssim \frac{\nu(\theta)}{(\theta+1)^\frac1{p}}+\sum_{k=\theta+1}^{n-1} \frac{\varepsilon_p(k)}{k}.
\end{align}
From \eqref{ccc} and \eqref{cdd}, it follows that $\text{(ii)}\Leftrightarrow\text{(iii)}$.
\qed

\section{Embedding Schramm spaces into $V_p[\nu]$}\label{sec7}

In this section, we shall describe relationships between $V_p[\nu]$ and various spaces of functions of generalized bounded variation. In this regard, $\Phi \text{BV}$ provides us with a suitable framework, by allowing for an extension of several such spaces (see Remark \ref{rem1}). 

We should remark in passing that \eqref{Hol} implies as well
$$
V[\nu(n)] \hookrightarrow V_p[\nu(n)] \hookrightarrow V\big[\nu(n)n^{1-\frac1{p}}\big],
$$
clarifying relations between $V_p[\nu]$ and Chanturiya classes.\footnote{Note of course that $\nu(n)n^{1-\frac1{p}}$ is not automatically concave. But this does not cause any problem as it is always quasiconcave (see \cite[Remark 2]{G}).}

The following lemma will be used in the proof of the sufficiency part of Theorem \ref{embedding}. This was first proven in \cite{Wu}. For each $n$, $\Phi_n^{-1}(x)$ denotes the inverse of the function $\Phi_n(x):=\sum\limits_{j=1}^n\phi_j(x)$, $x\geq0$.

\begin{lemma}\label{lemma1}
Let $1<p<\infty$ and $n\in\mathbb{N}$. If $f\in \Phi \emph{BV}$ and $\{x_j\}$ is a nonincreasing sequence of nonnegative real numbers such that
$$
\sum_{j=1}^n \phi_j(x_j)\leq \emph{Var}_\Phi(f),
$$
then
\begin{equation*}\label{conc}
\Big(\sum_{j=1}^n x_j^p\Big)^{\frac1{p}}\leq 16\max_{1\leq m\leq n} m^{\frac1{p}}\Phi_m^{-1}\big(\emph{Var}_\Phi(f)\big).
\end{equation*}
\end{lemma}

The proof of the following result uses techniques form \cite{G} and \cite{GHM}.
\begin{theorem}\label{embedding}
Let $\Phi$ be a $\Phi$-sequence, $\nu$ be a modulus of variation and $1 \leq p < \infty$. Then $\Phi\emph{BV}$ embeds into $V_p[\nu]$ if and only if
\begin{equation}\label{criterion}
\limsup_{n\rightarrow\infty} \frac1{\nu(n)} \max_{1 \leq k \leq n}k^{\frac1{p}}\Phi_k^{-1}(1)<\infty.
\end{equation}
\end{theorem}
\begin{proof}[Proof]
For simplicity, we shall assume throughout the proof that all functions are defined on $[0,1]$. 
Suppose \eqref{criterion} does not hold. Then there exist sequences $\{n_k\}$ and $\{m_k\}$ such that $n_k>2^{k+2}$ and,

\begin{equation}\label{ee1}
1 \leq m_k \leq n_k,
\end{equation}

\begin{equation}\label{e2}
\frac{m_k^{\frac1{p}}\Phi_{m_k}^{-1}(1)}{\nu(n_k)}>2^{4k},
\end{equation}
where
$$
\max_{1 \leq \gamma \leq n_k}\gamma^{\frac1{p}}\Phi_\gamma^{-1}(1)=m_k^{\frac1{p}}\Phi_{m_k}^{-1}(1).
$$

For each $k$, let $s_k$ be the largest integer such that $2s_k-1\leq 2^{-k}n_k$, and put $r_k:=\min\{m_k,s_k\}$. Then define
$$ f_k(x):=\begin{cases}
2^{-k}\Phi_{m_k}^{-1}(1), &\textmd{if}\  x\in [2^{-k}+\frac{2j-2}{n_k},2^{-k}+\frac{2j-1}{n_k}) ;~~~~~ 1\leq j\leq r_k, \\
0 ,  &\textmd{otherwise in} \ [0,1].\\
\end{cases}
$$
Since the $f_k$ have disjoint supports,
$$
f(x):=\sum_{k=1}^{\infty}f_k(x)
$$
is well-defined. Furthermore, one can without much difficulty verify that for each positive integer $k$,
\begin{equation}\label{2.3}
\text{Var}_\Phi(f_k)=\sum_{j=1}^{2r_k} \phi_j(2^{-k}\Phi_{m_k}^{-1}(1)),
\end{equation}
and
\begin{equation}\label{2.4}
\text{Var}_\Phi(f)\leq\sum_{k=1}^{\infty}\text{Var}_\Phi(f_k).
\end{equation}
Since for each $k$, $\Phi_{m_k}$ is an increasing convex function with $\Phi_{m_k}(0)=0$, it can be observed that
$$
\Phi_{m_k}(\alpha x)\leq\alpha\Phi_{m_k}(x)\ \ \ (0<\alpha<1,\ \ x\geq 0).
$$
By using this fact, along with \eqref{2.3} and \eqref{2.4}, we obtain
\begin{align*}
\text{Var}_\Phi(f) & \le\
\sum_{k=1}^{\infty} \sum_{j=1}^{2r_k} \phi_j(2^{-k}\Phi_{m_k}^{-1}(1))\\[.1in]
&=\sum_{k=1}^{\infty} \Phi_{2r_k}(2^{-k}\Phi_{m_k}^{-1}(1)) \\[.1in]
&\leq \sum_{k=1}^{\infty} \Phi_{2m_k}(2^{-k}\Phi_{m_k}^{-1}(1))\\[.1in]
&\leq\sum_{n=1}^{\infty} 2\Phi_{m_k}(2^{-k}\Phi_{m_k}^{-1}(1))<\infty.
\end{align*}
Therefore, $f \in \Phi\text{BV}$.

Let us now show that $f\notin V_p[\nu]$. First note that by the definition of $s_k$ we have $2(s_k+1)-1>2^{-k}n_k$. Combining this fact with inequality $n_k>2^{k+2}$ yields
\begin{equation}\label{ee2}
2s_k-1\geq 2^{-k-1}n_k.
\end{equation}
Therefore, if $r_k=s_k$, it follows from \eqref{ee1} and \eqref{ee2} that
$$
2r_k-1\geq 2^{-k-1}n_k \geq 2^{-k-1} m_k,
$$
and if $r_k=m_k$, we get $2r_k-1\geq m_k$. Thus in any case we have
\begin{equation}\label{e4}
(2r_k-1)^{\frac1{p}}\Phi_{m_k}^{-1}(1)\geq 2^{\frac{-k-1}{p}}m_k^{\frac1{p}}\Phi_{m_k}^{-1}(1).
\end{equation}
As a result, by letting
$$
I_j:=\left[2^{-k}+\frac{j-1}{n_k},2^{-k}+\frac{j}{n_k}\right], \hspace{.3in}1\leq j\leq2r_k-1,
$$
from \eqref{e4} we get
\begin{align}\label{e5}
\begin{split}
\upsilon_p(n_k,f)\geq \Big(\sum_{j=1}^{2r_k-1}|f(I_j)|^p\Big)^{\frac1{p}}=(2r_k-1)^{\frac1{p}}2^{-k}\Phi_{m_k}^{-1}(1)\\[.1in]
&\hspace{-4.1cm}\geq 2^{\frac{-kp-k-1}{p}}m_k^{\frac1{p}}\Phi_{m_k}^{-1}(1).
\end{split}
\end{align}
Finally, using \eqref{e2} and \eqref{e5} we conclude that
$$
\frac{\upsilon_p(n_k,f)}{\nu(n_k)}\geq \frac{2^{\frac{-kp-k-1}{p}}m_k^{\frac1{p}}\Phi_{m_k}^{-1}(1)}{\nu(n_k)}\geq 2^{\frac{-kp-k-1}{p}}\cdot2^{4k} \geq 2^k,
$$
which means $f\notin V_p[\nu]$.

\bigskip
Conversely, let $f\in \Phi\text{BV}$. Then there exists a constant $c>0$ such that $\text{Var}_\Phi(cf)<\infty$. Without loss of generality we may assume that $c=1$. Let $n\in\mathbb{N}$ and $\{I_j\}_{j=1}^n$ be a nonoverlapping collection of subintervals of $[0,1]$. We consider two cases:

\noindent {\bf Case I:} $p=1$. Since $\Phi_n$ is convex, from Jensen's inequality we get
\begin{equation}\label{1}
\Phi_n\left(\frac{\sum_{j=1}^n|f(I_j)|}{n}\right)\leq\sum_{j=1}^n\frac{\Phi_n(|f(I_j)|)}{n}.
\end{equation}
On the other hand, by rearranging the terms of the right hand side of (\ref{1}) we may write
\begin{align*}
\sum_{j=1}^n\Phi_n(|f(I_j)|)=\Big(\phi_1(|f(I_1)|)+\phi_2(|f(I_2)|)+...+\phi_n(|f(I_n)|)\Big)\\[.1in]
&\hspace{-8cm}+\Big(\phi_2(|f(I_1)|)+\phi_3(|f(I_2)|)+...+\phi_1(|f(I_n)|)\Big)\\[.1in]
&\hspace{-7.8cm}\vdots\\[.1in]
&\hspace{-8cm}+\Big(\phi_n(|f(I_1)|)+\phi_1(|f(I_2)|)+...+\phi_{n-1}(|f(I_n)|)\Big)\leq n\text{Var}_\Phi(f).
\end{align*}
Combining this with (\ref{1}) yields
$$
\Phi_n\left(\frac{\sum_{j=1}^n|f(I_j)|}{n}\right)\leq \text{Var}_\Phi(f).
$$
Since $\Phi_n^{-1}(0)=0$ for all $n$, making use of the concavity of $\Phi_n^{-1}$ one can easily verify that for all $\alpha,x\geq0$,
\begin{equation}\label{concave}
\Phi_n^{-1}(\alpha x)\leq(1+\alpha)\Phi_n^{-1}(x).
\end{equation}
Applying this fact with $\alpha=\text{Var}_\Phi(f)$ and $x=1$, an estimation of the sum $\sum\limits_{j=1}^n|f(I_j)|$ is obtained:
\begin{align*}\label{in1}
\sum_{j=1}^n|f(I_j)|&\leq \left(1+\text{Var}_\Phi(f)\right)n\Phi_n^{-1}\big(1\big)\\[.1in]
&=\left(1+\text{Var}_\Phi(f)\right)\max_{1 \leq k \leq n}k\Phi_k^{-1}(1).
\end{align*}

\noindent {\bf Case II:} $p>1$. Without loss of generality, we may assume that the $|f(I_j)|$ are arranged in descending order. Then taking $x_j:=|f(I_j)|$ in Lemma \ref{lemma1} and using  \eqref{concave} yields
$$
\Big(\sum_{j=1}^n |f(I_j)|^p\Big)^{\frac1{p}}\leq 16\left(1+\text{Var}_\Phi(f)\right)\max_{1\leq k\leq n} k^{\frac1{p}}\Phi_k^{-1}(1).
$$
Consequently, in each case $f\in V_p[\nu]$ as $n$ was arbitrary.
\end{proof}

With suitable choices of $\Phi$ in the above theorem we obtain characterizations of the embedding $X \hookrightarrow V_p[\nu]$, for various function spaces $X$ considered in the literature.

\begin{cor}\label{conf}
Let $\phi$ be an Orlicz function, $\Lambda$ be a $\Lambda$-sequence, and $1 \leq p,q < \infty$. Then

\begin{itemize}

\medskip
\item [\rm{(i)}]~ $BV_q \hookrightarrow V_p[\nu] \ \ \ \Longleftrightarrow \ \ \ \limsup\limits_{n\rightarrow\infty}\frac1{\nu(n)}\max\limits_{1 \leq k \leq n}k^{\frac1{p}-\frac1{q}}<\infty$;

\medskip
\item [\rm{(ii)}]~ $V_\phi \hookrightarrow V_p[\nu] \ \ \ \Longleftrightarrow \ \ \ \limsup\limits_{n\rightarrow\infty}\frac1{\nu(n)} \max\limits_{1 \leq k \leq n}k^{\frac1{p}}\phi^{-1}\Big(\frac1{k}\Big)<\infty$;

\medskip
\item [\rm{(iii)}]~ $\Lambda\emph{BV} \hookrightarrow V_p[\nu] \ \ \ \Longleftrightarrow \ \ \ \limsup\limits_{n\rightarrow\infty}\frac1{\nu(n)}\max\limits_{1 \leq k \leq n}\frac{k^{\frac1{p}}}{\sum_{j=1}^k \frac1{\lambda_j}}<\infty$;

\medskip
\item [\rm{(iv)}]~ $\Lambda\emph{BV}^{(q)} \hookrightarrow V_p[\nu] \ \ \ \Longleftrightarrow \ \ \ \limsup\limits_{n\rightarrow\infty}\frac1{\nu(n)}\max\limits_{1 \leq k \leq n}k^{\frac1{p}}\Big(\sum_{j=1}^k \frac1{\lambda_j}\Big)^{-\frac1{q}}<\infty$; and

\medskip
\item [\rm{(v)}]~ $\phi\Lambda\emph{BV} \hookrightarrow V_p[\nu] \ \ \ \Longleftrightarrow \ \ \ \limsup\limits_{n\rightarrow\infty}\frac1{\nu(n)}\max\limits_{1 \leq k \leq n}k^{\frac1{p}}\phi^{-1}\Big(\Big(\sum_{j=1}^k \frac1{\lambda_j}\Big)^{-1}\Big)<\infty$.

\end{itemize}

\end{cor}

As another consequence of Theorem \ref{embedding} we have the main result of \cite{G}.

\begin{cor}
Let $\Phi$ be a $\Phi$-sequence and $\nu$ be a modulus of variation. Then $\Phi\emph{BV}$ embeds into $V[\nu]$ if and only if
$$
\limsup_{n\rightarrow\infty} \frac{n\Phi_n^{-1}(1)}{\nu(n)}<\infty.
$$
\end{cor}

\section{Symmetric sequence spaces, generalized variation and Fourier series}\label{new}

In this section we discuss consequences of Theorem \ref{embedding} for certain symmetric sequence spaces. This connection is provided via a pivotal result of Berezhnoi \cite{Berezhnoi1} that translates embedding theorems between symmetric sequence spaces with the Fatou property to embedding theorems between their corresponding spaces of functions of generalized bounded variation, and vice versa. In the second part of this section we apply results from \cite{Berezhnoi2} to obtain a convergence result of Dirichlet--Jordan type for $V_p[\nu]$, as well as to estimate the Fourier coefficients of functions in $V_p[\nu]$.

\subsection{Embeddings between certain symmetric sequence spaces}
Given a sequence $x = \{x_j\}_{j=1}^{\infty}$ of real numbers, denote the nonincreasing rearrangement of the sequence $\{|x_j|\}_{j=1}^{\infty}$ by $x^{*}=\{x^{*}_{j}\}_{j=1}^{\infty}$. A Banach sequence space $X$ is called \textit{symmetric} if $x\in X$ and $y^*_j\leq x^*_j$ for all $j$, implies $y\in X$ and $\|y\|_X\leq \|x\|_X$.

The {\it dual symmetric sequence space to $X$} (see, e.g., \cite{BSh,KM}) is denoted by $X'$ and its norm is given by
\begin{equation*}\label{ddd}
\|y\|_{X'}:=\sup\Big\{\sum_{k=1}^\infty x_k y_k: \|x\|_{X}\leq1\Big\}.
\end{equation*}
A symmetric sequence space $X$ is said to have the {\it Fatou property} if $X=X''$.

Important examples of such symmetric sequence spaces include Marcinkiewicz, Lorentz and Orlicz sequence spaces:

\medskip
{\bf Marcinkiewicz space.} Let $\nu=\{\nu(n)\}$ be a modulus of variation. The Marcinkiewicz sequence space $m(\nu,p)$ is the Banach space of all real sequences $x=\{x_j\}$ for which
$$
\|x\|_{m(\nu,p)}:=\sup_{1\leq n<\infty}\frac1{\nu(n)}\Big(\sum_{j=1}^n x_j^p\Big)^\frac1{p}<\infty.\footnote{Although $m(\nu,p)$ is defined slightly different than what Berezhnoi \cite{Berezhnoi1} calls the Marcinkiewicz space, we still use the same terminology when referring to $m(\nu,p)$.}
$$

{\bf Lorentz space.} Let $w=\{w_j\}$ be a {\it weight sequence}, that is, a nonincreasing sequence of positive numbers such that $w\in c_0 \backslash \ell_1$.\footnote{A sequence $w=\{w_j\}$ is a weight if and only if $\{1/w_j\}$ is a $\Lambda$-sequence (Remark \ref{rem1}).} By definition, the Lorentz sequence space $d(w,q)$ is the Banach space of all real sequences $x=(x_j)$ for which
$$
\|x\|_{d(w,q)}:=\Big(\sum_{j=1}^\infty (x_j^*)^q w_j\Big)^\frac1{q}<\infty.
$$

\medskip
{\bf Orlicz space.} Let $\phi$ be an Orlicz function. The Orlicz sequence space $\ell_\phi$ is the Banach space of all real sequences $x=(x_j)$ for which
$$
\|x\|_{\ell_\phi}:=\inf\Big\{c>0:\sum_{j=1}^\infty\phi(|x_j|/c)\leq1\Big\}<\infty.
$$

An important generalization of the Orlicz space is the concept of

\medskip
{\bf Modular space.} Let $\Phi=\{\phi_j\}$ be a sequence of Orlicz functions. The modular sequence space $\ell_\Phi$ is the Banach space of all real sequences $x=\{x_j\}$ for which
$$
\|x\|_{\ell_\Phi}:=\inf\Big\{c>0:\sum_{j=1}^\infty\phi_j(|x_j|/c)\leq1\Big\}<\infty.
$$

\begin{remark}
By taking $X$ to be $m(\nu,p)$, $d(w,q)$, $\ell_\phi$ and $\ell_\Phi$, we observe that $BV(X)$ coincides with $V_p[\nu]$, $\Lambda BV^{(q)}$, $V_\phi$ and $\Phi BV$, respectively.
\end{remark}

With the aid of the norm of each symmetric sequence space $X$, one may define a certain type of variation for functions on an interval $[a,b]$.

\begin{definition}[\cite{Berezhnoi1}]
A function $f$ on $[a,b]$ is said to be of \textit{bounded $X$-variation} if
$$
Var_X(f):=\sup_{\{I_j\}}\|\{f(I_j)\}_j\|_X<\infty,
$$
where the supremum is taken over all collections $\{I_j\}$ of nonoverlapping subintervals of $[a,b]$.
\end{definition}
The set of all functions $f$ of bounded $X$-variation is denoted by $BV(X)$ which turns into a Banach space with the norm
$$
\|f\|_{BV(X)}:=Var_X(f)+\sup_{t\in[a,b]}|f(t)|.
$$

Let $X_0$ and $X_1$ be symmetric sequence spaces. In \cite[Theorem 1]{Berezhnoi1}, Berezhnoi established the following
\begin{theorem}\label{Berezh}
A necessary and sufficient condition for the embedding $X_0\hookrightarrow X_1$ to hold is that $BV(X_0)\hookrightarrow BV(X_1)$.
\end{theorem}
In view of this result combined with Theorem \ref{embedding}, we obtain the following characterization of the embedding of Modular spaces into Marcinkiewicz spaces which is, to the best of our knowledge, not proven or stated elsewhere in the literature. 

\begin{theorem}
Let $\Phi=\{\phi_j\}$ be a $\Phi$-sequence and $\nu$ be a modulus of variation. Then the embedding $\ell_\Phi\hookrightarrow m(\nu,p)$ holds if and only if \eqref{criterion} is fulfilled.
\end{theorem}

Likewise, from Corollary \ref{conf} and Theorem \ref{Berezh} we may, for instance, infer the following

\begin{cor}\label{conf}
Let $\phi$ be an Orlicz function and $w$ be a weight sequence. Then

\begin{itemize}

\medskip
\item [\rm{(i)}]~ $\ell_\phi \hookrightarrow m(\nu,p) \ \ \ \Longleftrightarrow \ \ \ \limsup\limits_{n\rightarrow\infty}\frac1{\nu(n)} \max\limits_{1 \leq k \leq n}k^{\frac1{p}}\phi^{-1}\Big(\frac1{k}\Big)<\infty$.

\medskip
\item [\rm{(ii)}]~ $d(w,q) \hookrightarrow m(\nu,p) \ \ \ \Longleftrightarrow \ \ \ \limsup\limits_{n\rightarrow\infty}\frac1{\nu(n)}\max\limits_{1 \leq k \leq n}k^{\frac1{p}}\Big(\sum_{j=1}^k w_j\Big)^{-\frac1{q}}<\infty$.
\end{itemize}

\end{cor}

\subsection{More on Fourier series and Fourier coefficients in $V_p[\nu]$}\label{sec6}

In this concluding part, we give a (uniform) convergence criterion for the Fourier series of functions in $V_p[\nu]$, and then obtain the order of magnitude of Fourier coefficients in this space.

It was proven in \cite[Theorem 8]{Berezhnoi2} that for the Fourier series of an arbitrary continuous function in $BV(X)$ to converge uniformly, it is necessary and sufficient that $\big\|\big\{\frac1{k}\big\}\big\|_{X'}<\infty$.

As $V_p[\nu]=BV(m(\nu,p))$, by invoking the above-mentioned result, the proof of Proposition \ref{Unif2} reduces to showing that
\begin{equation}\label{normm}
 \Big\|\Big\{\frac1{k}\Big\}\Big\|_{m(\nu,p)'}\approx \sum_{k=1}^\infty \frac{\nu(k)}{k^{1+\frac1{p}}}.   
\end{equation}

\begin{theorem}\label{Unif2}
Let $\nu$ be a modulus of variation and $1\leq p<\infty$. Then the following conditions are equivalent.

\begin{itemize}

\medskip
\item [\rm{(i)}]~ The Fourier series of every continuous function in $V_p[\nu]$ converges uniformly;

\medskip
\item [\rm{(ii)}]~ $\sum\limits_{k=1}^\infty \nu(k)k^{-(1+\frac1{p})}< \infty$;

\medskip
\item [\rm{(iii)}]~ $\sum\limits_{k=1}^\infty \Delta(k^{-\frac1{p}})\nu(k)< \infty$;

\medskip
\item [\rm{(iv)}]~ $\sum\limits_{k=1}^{\infty}\varepsilon_1(k)k^{-\frac{1}{p}}< \infty$; 

\medskip
\item [\rm{(v)}]~ $\sum\limits_{k=1}^{\infty}\varepsilon_p(k)k^{-1}< \infty$.

\end{itemize}
\end{theorem}
\begin{proof}
First observe that by similar arguments as in the proof of part (ii)$\Leftrightarrow$(iii) of Theorem \ref{convergence}, we have
\begin{equation*}
\sum_{k=1}^{n}\frac{\nu(k)-\nu(k-1)}{k^{-\frac{1}{p}}}=\sum\limits_{k=1}^n \Delta(k^{-\frac1{p}})\nu(k)+\frac{\nu(n)}{(n+1)^{\frac1{p}}}
\end{equation*}
and
\begin{equation*}
(1-2^{-\frac1{p}})\sum\limits_{k=1}^n\frac{\nu(k)}{k^{1+\frac1{p}}}\leq \sum_{k=1}^{n}\frac{\nu(k)-\nu(k-1)}{k^{-\frac{1}{p}}}-\frac{\nu(n)}{(n+1)^{\frac1{p}}}\leq p^{-1}\sum\limits_{k=1}^n\frac{\nu(k)}{k^{1+\frac1{p}}},
\end{equation*}
which imply (iii)$\Leftrightarrow$(iv) and (ii)$\Leftrightarrow$(iv), respectively.

Let $\|x\|_{m(\nu,p)}<1$, put $y_k:=\frac1{k}$ in \eqref{Abel}, and apply H\"{o}lder's inequality to obtain
\begin{align*}
\sum_{k=1}^n \frac{x_k}{k}=\sum_{k=1}^{n-1}\frac1{k(k+1)}\sum_{j=1}^k x_j+\frac1{n}\sum_{j=1}^n x_j\\[.1in]
&\hspace{-5.5cm}\leq\sum_{k=1}^{n-1}\frac1{k(k+1)}\Big(\sum_{j=1}^k (x_j^*)^p\Big)^{\frac1{p}}k^{1-\frac1{p}}+\frac1{n}\Big(\sum_{j=1}^n (x_j^*)^p\Big)^{\frac1{p}}n^{1-\frac1{p}}\\[.1in]
&\hspace{-5.5cm}\leq\sum_{k=1}^{n-1}\frac{\nu(k)}{k^{1+\frac1{p}}}+\frac{\nu(n)}{n^{\frac1{p}}}\leq \sum_{k=1}^\infty \frac{\nu(k)}{k^{1+\frac1{p}}},
\end{align*}
where the last inequality follows from the fact that
\begin{align*}
\frac{\nu(n)}{n^{\frac1{p}}}&=\lim_{m\rightarrow\infty}\nu(n)\Big(\frac1{n^{\frac1{p}}}-\frac1{m^{\frac1{p}}}\Big)\\[.1in]
&=\lim_{m\rightarrow\infty}\nu(n)\sum_{k=n}^m \Delta\Big(\frac1{k^{\frac1{p}}}\Big)\\[.1in]
&\leq \lim_{m\rightarrow\infty}\sum_{k=n}^m \Delta\Big(\frac1{k^{\frac1{p}}}\Big)\nu(k)\\[.1in]
&\leq \lim_{m\rightarrow\infty}\sum_{k=n}^m \frac{\nu(k)}{k^{1+\frac1{p}}}=\sum_{k=n}^\infty \frac{\nu(k)}{k^{1+\frac1{p}}}.
\end{align*}
Consequently,
\begin{equation*}
 \Big\|\Big\{\frac1{k}\Big\}\Big\|_{m(\nu,p)'}\leq \sum_{k=1}^\infty \frac{\nu(k)}{k^{1+\frac1{p}}}.   
\end{equation*}
Finally, it suffices to show that $\big\|\big\{\frac1{k}\big\}\big\|_{m(\nu,p)'}\gtrsim \sum\limits_{k=1}^\infty \frac{\nu(k)}{k^{1+\frac1{p}}}$ and (iv)$\Leftrightarrow$(v). Note that $\{\varepsilon_p(k)\}\in m(\nu,p)$ and $\|\{\varepsilon_p(k)\}\|_{m(\nu,p)}\leq 1$. Thus, $\big\|\big\{\frac1{k}\big\}\big\|_{m(\nu,p)'}\geq \sum\limits_{k=1}^\infty \frac{\varepsilon_p(k)}{k}$. On the other hand, using similar arguments resulting in \eqref{equi1}, \eqref{milad} and applying Lemma \ref{epsionprop}(iii), we obtain
\begin{align*}
\sum\limits_{k=1}^n \frac{\varepsilon_p(k)}{k}&\gtrsim \sum\limits_{k=1}^n\frac{\nu(k)-\nu(k-1)}{k^{\frac{1}{p}}}\\[.1in]
&\gtrsim\frac{\nu(n)}{(n+1)^{\frac1{p}}}+\sum\limits_{k=1}^n \frac{\nu(k)}{k^{1+\frac1{p}}}.
\end{align*}
This implies the validity of both \eqref{normm} and (i)$\Leftrightarrow$(v).
\end{proof}

Let $X$ be a symmetric sequence space. The {\it fundamental sequence} of $X$ is defined to be $\varphi_X(n):=\|\sum_{j=1}^n \{\delta_{ij}\}_{i=1}^\infty\|_X$, where $\delta_{ij}$ is the Kronecker delta. In order to obtain the order of magnitude of the Fourier coefficients in $V_p[\nu]$, we apply \cite[Theorem 14]{Berezhnoi2} which states that if $X$ is a symmetric sequence space with fundamental sequence $\varphi_X(n)$, then the Fourier coefficients $\hat{f}(n)$ of a function $f\in BV(X)$ satisfy the inequality
$$
|\hat{f}(n)|\lesssim \frac1{\varphi_X(n)}.
$$
By straightforward calculations, one observes that $\varphi_{m(\nu,p)}(n)=n^{\frac1{p}}/\nu(n)$. As a result, we have
\begin{prop}\label{coeff}
Let $\nu$ be a modulus of variation and $1\leq p<\infty$. If $f\in V_p[\nu]$ is $2\pi$-periodic, then 
$$
|\hat{f}(n)|\lesssim\frac{\nu(n)}{n^{\frac1{p}}}.
$$
\end{prop}

\medskip

\section*{Acknowledgements} The authors are grateful to the anonymous referee for his/her careful reading of the paper and making several helpful comments which have significantly improved the quality of this paper. Thanks are also extended to Prof. E.I. Berezhnoi for a useful discussion. The second named author was supported by the Iran National Science Foundation [project number 97013049]. The third named author is supported by a grant from IPM. Part of this paper was written when he was visiting the Department of Mathematics of IASBS. He wishes to express his sincere thanks to the Department, A. Ghorbanalizadeh and M.M. Sadr for their hospitality during the visit.

\end{document}